\documentclass[12pt,reqno]{amsart}
\usepackage{fullpage}
\usepackage{amssymb,amsfonts}
\usepackage[all,arc]{xy}
\usepackage{bbm}
\usepackage{enumerate}
\usepackage{mathrsfs}
\usepackage{color}   %May be necessary if you want to color links
\usepackage{hyperref}
\hypersetup{
    colorlinks=true, %set true if you want colored links
    linktoc=all,     %set to all if you want both sections and subsections linked
    linkcolor=blue,  %choose some color if you want links to stand out
}
\usepackage{amsmath}
\usepackage{amsthm}
\usepackage{graphicx}
\usepackage{wasysym}
\usepackage{pgf,tikz}
\usepackage{mathtools}
\usetikzlibrary{positioning}

\newcommand{\bbC}{\mathbb C}

\newcommand{\bbH}{\mathbb H}

\newcommand{\bbR}{\mathbb R}

\newcommand\floor[1]{\lfloor#1\rfloor}

\newcommand{\fg}{\mathfrak g}
\newcommand{\fk}{\mathfrak k}

\newcommand{\fh}{\mathfrak h}

\newtheorem{Thm}{Theorem}[section]
\newtheorem{Prop}[Thm]{Proposition}
\newtheorem{Lem}[Thm]{Lemma}

\theoremstyle{definition}

\newtheorem{Rem}{Remark}

\theoremstyle{remark}

\theoremstyle{definition}

\numberwithin{equation}{section}

\title{Logarithm laws for unipotent flows on hyperbolic manifolds}

\author{Shucheng Yu}

\address{Department of Mathematics, Boston College, Chestnut Hill MA 02467-3806, USA}

\email{shucheng.yu@bc.edu}
\begin{document}
\begin{abstract}
We prove logarithm laws for unipotent flows on non-compact finite-volume hyperbolic manifolds. Our method depends on the estimate of norms of certain incomplete Eisenstein series.
\end{abstract}

\maketitle
\section{Introduction}
Let $G$ denote a connected real semisimple Lie group with no compact factors and $\Gamma\subset G$ be a non-uniform irreducible lattice, that is, $\Gamma$ is discrete, the homogeneous space $\Gamma\backslash G$ is non-compact and has finite co-volume with respect to the Haar measure of $G$. Let $\sigma$ denote the right $G$-invariant probability measure on $\Gamma\backslash G$. Any unbounded one-parameter subgroup $\{g_t\}_{t\in \bbR}\subset G$ acts on $\Gamma\backslash G$ by right multiplication. By Moore's Ergodicity Theorem this action is ergodic with respect to $\sigma$, hence for $\sigma$-a.e. $x\in\Gamma\backslash G$ the orbit $\{xg_t\}$ is dense. In particular, these orbits will make excursions into the cusp(s) of $\Gamma\backslash G$. A natural question to ask is at what rate these cusp excursions occur.

We first fix some notations throughout this paper. We write $A\asymp B$ if there is some constant $c>1$ such that $c^{-1}A\leq B\leq cA.$ And we write $A\lesssim B$ or $A= O\left(B\right)$ to indicate that $A\leq cB$ for some positive constant $c$. We will use subscripts to indicate the dependence of the constant on some parameters.

The above question can be restated as a shrinking target problem. Let $K$ be a maximal compact subgroup of $G$, $\Gamma\backslash G$ has a naturally defined distance function, $\textrm{dist}$, induced from a left $G$-invariant and bi-$K$-invariant Riemannian metric on $G$. For a fixed $o\in \Gamma\backslash G$ we define the cusp neighborhoods by
$$B_r:=\left\{x\in \Gamma\backslash G\ |\ \textrm{dist}\left(o,x\right)> r\right\}$$
for any $r>0$. By \cite{KM} there exists a constant $\varkappa>0$ such that $\sigma\left(B_r\right)\asymp e^{-\varkappa r}$. For $\left\{r_{\ell}\right\}$ a sequence of positive real numbers with $r_{\ell}\to\infty$, consider the family of shrinking cusp neighborhoods $\left\{B_{r_{\ell}}\right\}$, we define a corresponding sequence of random variables on $\Gamma\backslash G$ by
\[X_{\ell}\left(x\right):= \left\{
  \begin{array}{lr}
    1 & \textrm{if}  \ \ xg_{\ell}\in B_{r_{\ell}}\\
    0 &  \textrm{otherwise.}
  \end{array}
\right.
\]
Note that $X_{\ell}\left(x\right)=1$ if and only if the $\ell$-th orbit of $x$ makes excursion into the $\ell$-th cusp neighborhood $B_{r_{\ell}}$. In this setting, one can vary the sequence $\{r_{\ell}\}$ to enlarge or shrink the family of cusp neighborhoods $\{B_{r_{\ell}}\}$, and then ask whether the events $X_{\ell}\left(x\right)=1$ happen finitely or infinitely many times for a generic point $x$. We note that the first half of Borel-Cantelli lemma implies that if $\sum_{\ell=1}^{\infty}\sigma\left(B_{r_{\ell}}\right)<\infty$, then for $\sigma$-a.e. $x\in \Gamma\backslash G$ the events $X_{\ell}\left(x\right)=1$ happen for finitely many $\ell$. Thus $\limsup_{\ell\to\infty}\frac{\textrm{dist}\left(o,xg_{\ell}\right)}{r_{\ell}}\leq 1$ for $\sigma$-a.e. $x\in \Gamma\backslash G$. In particular, for any small positive number $\epsilon$, choosing $r_{\ell}=\frac{\left(1+\epsilon\right)\log\left(\ell\right)}{\varkappa}$, a standard continuity argument implies that $\limsup_{t\to\infty}\frac{\textrm{dist}\left(o,xg_t\right)}{\log t}\leq \frac{1+\epsilon}{\varkappa}$ for $\sigma$-a.e. $x\in \Gamma\backslash G$. Letting $\epsilon\to0$ we get $\limsup_{t\to\infty}\frac{\textrm{dist}\left(o, xg_t\right)}{\log \left(t\right)}\leq \frac{1}{\varkappa}$ for $\sigma$-a.e. $x\in\Gamma\backslash G$. If the bound is sharp, $$\limsup_{t\to\infty}\frac{\textrm{dist}\left(o, xg_t\right)}{\log \left(t\right)}= \frac{1}{\varkappa},$$
for $\sigma$-a.e $x\in\Gamma\backslash G$, we say that the flow $\{g_t\}_{t\in\bbR}$ satisfies the \textit{logarithm law}. Following \cite{JA}, we say a sequence of cusp neighborhoods $\{B_{r_{\ell}}\}$ is \textit{Borel-Cantelli} for $\{g_t\}$ if $\sum_{\ell=1}^{\infty}\sigma\left(B_{r_{\ell}}\right)=\infty$ and for $\sigma$-a.e. $x\in \Gamma\backslash G$ $X_{\ell}(x)=1$ for infinitely many $\ell$. Note that $\left\{g_t\right\}$ satisfying logarithm law is equivalent to the statement that for any $\epsilon>0$, any sequence of cusp neighborhoods $\{B_{r_{\ell}}\}$ with $\sigma\left(B_{r_{\ell}}\right)\asymp \frac{1}{\ell^{1-\epsilon}}$ is Borel-Cantelli for $\{g_t\}$.

The problem of logarithm laws in the context of homogeneous space was first studied by Sullivian \cite{Su} where he proved logarithm laws for geodesic flows on non-compact finite-volume hyperbolic manifolds. The general case of one-parameter diagonalizable flows on non-compact finite-volume homogeneous spaces was proved by Kleinbock and Margulis \cite{KM}. The main ingredient of their proof is the exponential decay of matrix coefficients of diagonalizable flows, from which they deduced the quasi-independence of the above events $X_{\ell}$. Then the logarithm law follows from a quantitative Borel-Cantelli lemma.

The problem of logarithm laws for unipotent flows is more subtle since the matrix coefficients of unipotent flows only decay polynomially. Nevertheless, using a random analogy of Minkowski's theorem Athreya and Margulis \cite{AM} proved logarithm laws for one-parameter unipotent subgroups on the space of lattices $\mathcal{X}_d:=SL_d\left(\mathbb{Z}\right)\backslash SL_d\left(\bbR\right)$. Later Kelmer and Mohanmmadi \cite{DA} proved the case when $G$ is a product of copies of $SL_2\left(\bbR\right)$ and $SL_2\left(\bbC\right)$ and $\Gamma$ is any irreducible non-uniform lattice. We note that in the above two cases, their methods are closely related and both rely on the estimate of $L^2$-norms of certain transform functions.

In \cite{JA1}, Athreya studied the cusp excursion of the full horospherical group with respect to some one-parameter diagonalizable subgroup on $\mathcal{X}_d$. Surprisingly, he was able to relate the cusp excursion rates for diagonalizable and horospherical actions and certain Diophantine properties for every $x\in \mathcal{X}_d$. In particular, his result implies logarithm laws for unipotent flows on $\mathcal{X}_2$. The most general result known for unipotent flows was obtained by Athreya and Margulis \cite{AM2}. More precisely, for $G$ a semisimple Lie group without compact factors, $\Gamma\subset G$ an irreducible non-uniform lattice in $G$ and $\{g_t\}_{t\in\bbR}$ a one-parameter unipotent subgroup in $G$, they proved that for any $o\in \Gamma\backslash G$ and $\sigma\textrm{-a.e.}$ $x\in \Gamma\backslash G$, there exists $0<\beta\leq 1$ such that $\limsup_{t\to\infty}\frac{\textrm{dist}\left(o,xg_t\right)}{\log\left(t\right)}=\frac{\beta}{\varkappa}$. Moreover, they asked whether such $\beta$ can always attain $1$, which is the upper bound coming from the first half of Borel-Cantelli lemma.

In this paper, we generalize the approach in \cite{AM} and \cite{DA} to give a positive answer to this question when $\Gamma\backslash G$ is the frame bundle of hyperbolic manifolds. Before stating our main result, we first fix some notations. Let $\bbH^{n+1}$ be the $\left(n+1\right)$-dimensional real hyperbolic space with $n\geq 2$ and $\textrm{Iso}^{+}\left(\bbH^{n+1}\right)$ denote the orientation preserving isometry group of $\bbH^{n+1}$. Fix a maximal compact subgroup $K$ and identify $G/K$ with $\mathbb{H}^{n+1}$.

\begin{Thm}
\label{main thm}
Let $G=\text{Iso}^{+}\left(\bbH^{n+1}\right)$ with $n\geq 2$, $\Gamma\subset G$ a non-uniform lattice and $\{g_t\}_{t\in\bbR}$ a one-parameter unipotent subgroup of $G$. Let $\textrm{dist}\left(\cdot,\cdot\right)$  denote the distance function obtained from hyperbolic metric on the hyperbolic manifold $\Gamma\backslash \mathbb{H}^{n+1}$. Then for any fixed $o\in \Gamma\backslash G$,
\begin{equation}
\label{l. laws} \limsup\limits_{t\to \infty}\frac{\textrm{dist}\left(o,xg_t\right)}{\log t}= \frac{1}{n},
\end{equation}
for $\sigma$-a.e. $x\in\Gamma\backslash G$.\footnotemark
\end{Thm}
\footnotetext{Here by abuse of notation, for $o,xg_t\in \Gamma\backslash G$, we write $\textrm{dist}\left(o,xg_t\right)$ for the distance between their projections to $\Gamma\backslash \mathbb{H}^{n+1}$.}

We give a brief outline of our proof here. We first note that if $\left(\ref{l. laws}\right)$ holds for $\Gamma$, then it also holds for any $\Gamma'$ conjugate to $\Gamma$ (see section \ref{loglaw}). Hence after suitable conjugation we can assume $\Gamma$ has a cusp at $\infty$ (see section \ref{grt} for the definition of cusps).

The upper bound, as mentioned above, follows from the first half of the Borel-Cantelli lemma. For the lower bound, we first note that it suffices to show that the set
$$\mathcal{A}_{\epsilon}:=\left\{x\in \Gamma\backslash G\ |\ \limsup\limits_{t\to \infty}\frac{\textrm{dist}\left(o, xg_t\right)}{\log t}\geq \frac{1-\epsilon}{n}\right\}$$
has positive measure for any $\epsilon>0$. This is because $\mathcal{A}_{\epsilon}$ is invariant under the action of $\{g_t\}_{t\in\bbR}$, hence by ergodicity, if $\mathcal{A}_{\epsilon}$ is of positive measure it must have full measure. Then the theorem follows by letting $\epsilon$ approach zero.

Next, in order to show that $\mathcal{A}_{\epsilon}$ has positive measure, we construct a subset $\mathcal{B}_{\epsilon}\subset \mathcal{A}_{\epsilon}$ which we show has positive measure. To describe our construction, we need some additional notations. Fix an Iwasawa decomposition $G=NAK$ with the maximal unipotent subgroup $N$ fixing $\infty$. Let $M\subset K$ be the centralizer of $A$ in $K$, $P=NAM$ the stabilizer of $\infty$ in $G$ and $\Gamma_{\infty}=\Gamma\cap P$ the stabilizer of $\infty$ in $\Gamma$. Let $Q=NM$ be the maximal subgroup of $P$ containing $\Gamma_{\infty}$ such that $\Gamma_{\infty}\backslash Q$ is relatively compact. See section \ref{v.group} for explicit descriptions of these groups. For any $\mathfrak{D}\subset Q\backslash G$ we let
$$Y_{\mathfrak{D}}=\left\{ \Gamma g\in \Gamma\backslash G\ \big|\ Q\gamma g\in \mathfrak{D}\ \text{for some}\ \gamma\in \Gamma\right\}.$$
In section \ref{tl}, for any $\epsilon>0$ we construct a sequence of sets $\mathfrak{D}_{m}\subset Q\backslash G$ explicitly by taking unions of certain translations of cusp neighborhoods and we show that $\left\{\sigma\left(Y_{\mathfrak{D}_m}\right)\right\}_{m\in\mathbb{N}}$ is uniformly bounded from below and each $Y_{\mathfrak{D}_m}$ satisfies
\begin{equation}
\label{intro}\forall x\in Y_{\mathfrak{D}_m}\  \exists\ \ell\geq m\ \textrm{such that}\ \frac{\textrm{dist}\left(o,xg_{\ell}\right)}{\log \ell}\geq \frac{1-\epsilon}{n}.
\end{equation}

By $\left(\ref{intro}\right)$ it is clear that the limit superior set $\mathcal{B}_{\epsilon}:=\cap_{l=1}^{\infty}\cup_{m=l}^{\infty}Y_
{\mathfrak{D}_m}$ is contained in $\mathcal{A}_{\epsilon}$. Moreover, since $\{\sigma\left(Y_{\mathfrak{D}_m}\right)\}_{m\in\mathbb{N}}$ is uniformly bounded from below, $\mathcal{B}_{\epsilon}$ has positive measure. Hence $\mathcal{A}_{\epsilon}$ is of positive measure.

To show that $\left\{\sigma\left(Y_{\mathfrak{D}_m}\right)\right\}$ has a uniform lower bound, we find nice subsets $\mathfrak{D}_m'\subset \mathfrak{D}_m$ with $|\mathfrak{D}_m'|\asymp |\mathfrak{D}_m|$ (here $|\cdot|$ denotes a right $G$-invariant measure on $Q\backslash G$) and we show $\left\{\sigma\left(Y_{\mathfrak{D}'_m}\right)\right\}$ is uniformly bounded from below. One standard way to handle $\sigma\left(Y_{\mathfrak{D}_m'}\right)$ is to use the incomplete Eisenstein series to relate $\sigma\left(Y_{\mathfrak{D}_m'}\right)$ to $|\mathfrak{D}_m'|$. More precisely, for any compactly supported function $f$ on $Q\backslash G$ the corresponding incomplete Eisenstein series $\Theta_f\in L^2\left(\Gamma\backslash G\right)$ attached to $f$ is defined by
$$\Theta_f\left(g\right)= \sum_{\gamma\in \Gamma_{\infty}\backslash \Gamma}f\left(\gamma g\right).$$
Note that if $f$ is supported on $\mathfrak{D}$, then $\Theta_f$ is supported on $Y_{\mathfrak{D}}$. To show that $\left\{\sigma\left(Y_{\mathfrak{D}_m'}\right)\right\}$ is bounded from below, it is enough to show that the $L^2$-norm (with respect to the measure $\sigma$) of the incomplete Eisenstein series $\Theta_{\mathbbm{1}_{\mathfrak{D}_m'}}$ is not too large compared to the measure of $\mathfrak{D}_m'$, where $\mathbbm{1}_{\mathfrak{D}_m'}$ is the characteristic function of $\mathfrak{D}_m'$. To show this, we bound $||\Theta_{\mathbbm{1}_{\mathfrak{D}_m'}}||_2^2$ in terms of $|\mathfrak{D}_m'|$. In fact, for any parameter $\lambda>0$ we define a family of functions $\mathscr{A}_{\lambda}\subset L^2\left(Q\backslash G\right)$ (see description of $\mathscr{A}_{\lambda}$ in section \ref{proof1.2}) and we prove the following bound for functions in $\mathscr{A}_{\lambda}$.

\begin{Thm}
\label{thm 2} Let $G=\text{Iso}^{+}\left(\bbH^{n+1}\right)$ and $\Gamma\subset G$ a non-uniform lattice with a cusp at $\infty$. For any parameter $\lambda>0$ there exists some constant $C$ (depending on $\Gamma$ and $\lambda$) such that
\begin{equation}
\label{kb}||\Theta_f||_2^2\leq C\left(||f||_1^2+ ||f||_2^2\right)
\end{equation}
for any $f\in \mathscr{A}_{\lambda}$, where the norms on the right are with respect to the right $G$-invariant measure on $Q\backslash G$.
\end{Thm}

Our construction of $\mathfrak{D}_m'$ yields that we can take functions in $\mathscr{A}_{\lambda}$ (for some $\lambda$) to approximate $\mathbbm{1}_{\mathfrak{D}_m'}$, then we can use Theorem \ref{thm 2} to bound $||\Theta_{\mathbbm{1}_{\mathfrak{D}_m'}}||_2^2$ in terms of $|\mathfrak{D}_m'|$.

We note that our strategy of proving $\left(\ref{kb}\right)$ is similar to the one used in \cite{KM}. To prove Theorem \ref{thm 2}, we work out an explicit constant term formula for certain non-spherical Eisenstein series (for arbitrary $n\geq 2$). With this constant term formula, a formal computation ensures that we can bound the $L^2$-norm of any incomplete Eisenstein series by the right-hand side of $\left(\ref{kb}\right)$, together with a third term expressed in terms of the exceptional poles of Eisenstein series. Thus $\left(\ref{kb}\right)$ follows if we can bound this third term by the right-hand side of $\left(\ref{kb}\right)$. However, to prove this bound, we need to assume the functions are from $\mathscr{A}_{\lambda}$.

\begin{Rem}
An interesting question is whether $\left(\ref{kb}\right)$ holds uniformly for any $f\in C_c^{\infty}\left(Q\backslash G\right)$. In particular, for our purpose if one can prove $\left(\ref{kb}\right)$ uniformly for linear combinations of nonnegative functions in $\mathscr{A}_{\lambda}$, then the same method implies a stronger Borel-Cantelli law: every sequence of nested cusp neighborhoods $\{B_{r_{\ell}}\}$ with $\sum_{\ell=1}^{\infty}\sigma\left(B_{r_{\ell}}\right)=\infty$ is Borel-Cantelli for unipotent flows. Such a result was obtained in \cite[Remark 8]{DA} by proving $\left(\ref{kb}\right)$ for all nonnegative functions in $C_c^{\infty}\left(Q\backslash G\right)$ when $G$ is a product of copies of $SL_2\left(\bbR\right)\left(\cong\textrm{Iso}^+\left(\bbH^2\right)\right)$ and $SL_2\left(\bbC\right)\left(\cong \textrm{Iso}^+\left(\bbH^3\right)\right)$ and $\Gamma$ is any arithmetic irreducible lattice. Their proof of $\left(\ref{kb}\right)$ is indirect and depends on the existence of a family of lattices for which the Eisenstein series have no exceptional poles. We note that in \cite{VG} Gritsenko gave an example of such a lattice in $\textrm{Iso}^+\left(\bbH^{4}\right)$. Hence using the general constant term formula we get, one can prove the above Borel-Cantelli law (for unipotent flows) for this specific lattice (and its commensurable lattices) in $\textrm{Iso}^+\left(\bbH^4\right)$.
\end{Rem}

\begin{Rem}
We end the introduction by remarking that the sets $\mathfrak{D}_m$ are constructed by taking unions of translations of the neighborhoods at the cusp $\infty$ (along the unipotent flow), hence our method implies a slightly stronger result: logarithm laws for excursions of unipotent flows into any individual cusp (for other cusps, the result can be obtained by conjugating this cusp to $\infty$).
\end{Rem}

\subsection*{Acknowledgements}
I am very grateful to my advisor Dubi Kelmer for his guidance and all the discussions and useful comments. I would like to thank the anonymous referee for many helpful comments that made this paper more readable. This work is partially supported by NSF grant DMS-1401747.

\section{Preliminaries and Notations}
\subsection{Vahlen group}\label{v.group}
Let $\bbH^{n+1}$ denote the $\left(n+1\right)$-dimensional real hyperbolic space and $G= \text{Iso}^{+}\left(\bbH^{n+1}\right)$ be its orientation preserving isometry group. There are various hyperbolic models of $\bbH^{n+1}$ and each model gives an explicit description of $G$. In this paper, we choose the upper half space model and realize $G$ via the Vahlen group (see\cite{La 1} \cite{EGM 87} and \cite{EGM 90} for more details about Vahlen group).

We first briefly recall some facts about Clifford algebra. The Clifford algebra $C_n$ is an associative algebra over $\bbR$ with $n$ generators $e_1, \cdots, e_n$ satisfying relations $e_i^2= -1, e_ie_j= -e_je_i, i\neq j$. Let $\mathcal{P}_n$ be the set of subsets of $\{1,\cdots,n\}$. For $I\in\mathcal{P}_n$, $I=\{i_1,\cdots,i_r\}$ with $i_1< \cdots<i_r$ we define $e_I:=e_{i_1}\cdots e_{i_r}$ and $e_{\emptyset}=1$. These $2^n$ elements $e_I\left(I\in\mathcal{P}_n\right)$ form a basis of $C_n$. The Clifford algebra $C_n$ has a main anti-involution $^{\ast}$ and a main involution $'$. Explicitly, their actions on the basis elements are given by $\left(e_{i_1}\cdots e_{i_r}\right)^{\ast}= e_{i_r}\cdots e_{i_1}$ and $\left(e_{i_1}\cdots e_{i_r}\right)'= \left(-1\right)^r e_{i_1}\cdots e_{i_r}.$ Their composition $\overline{e_{I}}:= \left(e_{I}'\right)^{\ast}$ gives the conjugation map on $C_n$.

For any $1\leq i\leq n$, let $\mathbb{V}^{i}$ denote the real vector space spanned by $1, e_1, \cdots, e_i$. Note that $\textrm{dim}\mathbb{V}^i= i+1$. The Clifford group $T_{i}$ is defined to be the collection of all finite products of non-zero elements from $\mathbb{V}^{i}$ with group operation given by multiplication. There is a well-defined norm on $\mathbb{V}^{n}$ given by $|v|= \sqrt{v\bar{v}}$ and it extends multiplicatively to a norm on $T_{n}$.

In this setting, the $\left(n+1\right)$-dimensional hyperbolic space model is the upper half space
\begin{equation}
\label{Hn}\mathbb{H}^{n+1}:= \left\{x_0+ x_1e_1+ \cdots+ x_ne_n\in \mathbb{V}^n\ \big|\ x_i\in \mathbb{R}, x_n> 0\right\}
\end{equation}
endowed with the Riemannian metric
\begin{equation}
\label{hm}ds^2=\frac{dx_0^2+ \cdots + dx_n^2}{x_n^2}.
\end{equation}
Let $M_2\left(C_n\right)$ be the set of $2 \times 2$ matrices over $C_n$. The Vahlen group $SL\left(2, T_{n-1}\right)$ is defined by
\begin{displaymath}
 SL\left(2, T_{n-1}\right)=\left\{ \begin{pmatrix}
a& b\\
c& d
\end{pmatrix}\in M_2\left(C_n\right)\ \bigg|\ \begin{array}{ll}
a, b, c, d\in T_{n-1}\cup \{0\},\\
ab^{\ast}, cd^{\ast}\in \mathbb{V}^{n-1},\\
ad^{\ast}-bc^{\ast}=1
\end{array} \right\}.
\end{displaymath}

An element $g= \begin{pmatrix}
a& b\\
c& d
\end{pmatrix}$ in $SL\left(2, T_{n-1}\right)$ acts on $\bbH^{n+1}$ as an isometry via the M$\ddot{\textrm{o}}$bius transformation
\begin{equation}
\label{action} g\cdot v= \left(av+ b\right)\left(cv+ d\right)^{-1}.
\end{equation}
This gives a surjective homomorphism from $SL\left(2, T_{n-1}\right)$ to $\textrm{Iso}^{+}\left(\mathbb{H}^{n+1}\right)$ with kernel $\pm I_2$. Hence $G$ is realized as $PSL\left(2, T_{n-1}\right):=SL\left(2, T_{n-1}\right)/\{\pm I_2\}$. Here $I_2$ is the $2\times 2$ identity matrix

We fix an Iwasawa decomposition
$$PSL\left(2, T_{n-1}\right)= NAK,$$ with
$$N=\left\{u_{\textbf{x}}=\begin{pmatrix}
1& \textbf{x}\\
0& 1
\end{pmatrix}\ \big|\ \textbf{x}= x_0+x_1e_1+ \cdots+ x_{n-1}e_{n-1}\in \mathbb{V}^{n-1}\right\},$$
$$A= \left\{a_t=\begin{pmatrix}
e^{\frac{t}{2}}& 0\\
0& e^{-\frac{t}{2}}
\end{pmatrix}\ \big|\ t\in \bbR\right\},$$ and
$$K=\left\{g\in SL\left(2, T_{n-1}\right)\ \big|\ g\cdot e_n= e_n\right\}/\{\pm I_2\}$$
is the stabilizer of $e_n$. From this we can identify $G/K$ with $\mathbb{H}^{n+1}$ by sending $gK$ to $g\cdot e_n$.

An element in $K$ is of the form
$\begin{pmatrix}
q_2'&-q_1'\\
q_1&q_2
\end{pmatrix}$ with $|q_1|^2+|q_2|^2= 1$ and $q_1{q_2}^{\ast}\in \mathbb{V}^{n-1}$. Let $M$ be the centralizer of $A$ in $K$, that is, $M$ is the subgroup of $K$ consisting of diagonal matrices. For later use we note that $K$ is isomorphic to $SO\left(n+1\right)$, $M$ is isomorphic to $SO\left(n\right)$ and $M\backslash K$ can be identified with the $n$-sphere $S^n$ via the map
\begin{align*}
M\backslash K &\longrightarrow S^n:=\left\{x_0+x_1e_1+\cdots x_ne_n\ |\ x_i\in\bbR, x_0^2+x_1^2+\cdots x_n^2=1\right\}\\
\begin{pmatrix}
q_2' & -q_1'\\
q_1 & q_2\end{pmatrix}&\longmapsto \qquad\qquad\qquad 2\overline{q_1}q_2+ \left(|q_2|^2-|q_1|^2\right)e_n.
\end{align*}
Here this map is well-defined since $q_1q_2^{\ast}\in\mathbb{V}^{n-1}$ if and only if $\overline{q_1}q_2\in \mathbb{V}^{n-1}$ (\cite[Corollary 7.15]{JP}).

\subsection{Coordinates and normalization.}
Let $G=NAK$ be the fixed Iwasawa decomposition as above and let $Q=NM$. Under the coordinates $g= u_{\textbf{x}}a_tk$, the Haar measure of $G$ is given by
$$dg= e^{-nt}d\textbf{x}dtdk,$$
where $d\textbf{x}$ is the usual Lebesgue measure on $N$ (identified with $\bbR^{n}$), $dk$ is the probability Haar measure on $K$. Hence the probability Haar measure $\sigma=\sigma_{\Gamma}$ on $\Gamma\backslash G$ is given by
\begin{equation}
\label{haar}d\sigma\left(g\right)=\frac{1}{\nu_{\Gamma}}e^{-nt}d\textbf{x}dtdk
\end{equation}
with $\nu_{\Gamma}=\int_{\mathcal{F}_{\Gamma}}dg$, where $\mathcal{F}_{\Gamma}$ is a fundamental domain for $\Gamma\backslash G$.

We normalize the measures on various spaces as following. First for $\phi\in L^2\left(M\backslash K\right)$, we view $\phi$ as a left $M$-invariant function on $K$ and normalize the Haar measure on $M\backslash K$ (also denoted by $dk$) such that
$$\int_{M\backslash K}\phi\left(k\right)dk= \int_{K}\phi\left(k\right)dk.$$
Next we identify $Q\backslash G= A\times M\backslash K$ and we normalize the Haar measure on $Q\backslash G$ so that for any $f\in C^{\infty}_c\left(Q\backslash G\right)$ we have
\begin{equation}
\label{Q}\int_{Q\backslash G}f\left(g\right)dg= \int_{\bbR}\int_{M\backslash K}f\left(a_tk\right)e^{-nt}dtdk.
\end{equation}
We then normalize the Haar measure on $Q$ so that for any $f\in C^{\infty}_c(G)$ we have
\begin{equation}
\int_Gf\left(g\right)dg= \int_{Q\backslash G}\int_{Q}f\left(qg\right)dqdg.
\end{equation}

\subsection{Cusps and reduction theory}\label{grt}
Fix the notations as above. Let
$$\partial \mathbb{H}^{n+1}:=\left\{x_0+ x_1e_1+\cdots+ x_{n-1}e_{n-1}\in\mathbb{V}^{n-1}\ |\ x_i\in\bbR\right\}\cup\left\{\infty\right\}$$
denote the boundary of $\mathbb{H}^{n+1}$. The action (\ref{action}) extends naturally to $\partial\mathbb{H}^{n+1}$ by the same formula. Let $P=NAM$ be the subgroup of upper triangular matrices in $G$. We note that $P$ is the stabilizer of $\infty$. Let $\Gamma \subseteq G$ be a non-uniform lattice in $G$. Define
 $$\Gamma_{\infty}= \Gamma\cap P$$ and
 $$\Gamma_{\infty}'= \Gamma\cap N.$$
Note that $\Gamma_{\infty}$ is the stabilizer of $\infty$ in $\Gamma$, and $\Gamma_{\infty}'$ consists the identity and unipotent elements in $\Gamma_{\infty}$. We say that $\Gamma$ has a cusp at $\infty$ if $\Gamma_{\infty}'$ is nontrivial. Note that if $\Gamma$ has a cusp at $\infty$, then $\Gamma\backslash G$ having finite co-volume implies that $\Gamma_{\infty}'$ is a lattice (free abelian and of full rank) in $N$ (see \cite[Definition $0.5$ and Theorem $0.7$]{Ga}). Moreover, we note that by discreteness, $\Gamma_{\infty}\cap A=\{I_2\}$. Thus the conjugation action of $\Gamma_{\infty}$ on $N\left(=\mathbb{V}^{n-1}\right)$ and $\Gamma_{\infty}'$ induces an injection
$$\Gamma_{\infty}/\Gamma_{\infty}'\hookrightarrow SO\left(\mathbb{V}^{n-1}\right)\cap GL\left(\Gamma_{\infty}'\right).$$
Hence $\Gamma_{\infty}'$ is a finite index subgroup of $\Gamma_{\infty}$. Denote by $[\Gamma_{\infty}: \Gamma_{\infty}']$ this index.

For any $\xi\in \partial \mathbb{H}^{n+1}$, there exists some $g\in G$ such that $g\cdot \xi=\infty$. We say $\Gamma$ has a cusp at $\xi$ if $g\Gamma g^{-1}$ has a cusp at $\infty$. And we say two cusps $\xi, \xi'$ are $\Gamma$-equivalent if there exists some $\gamma\in \Gamma$ such that $\gamma\cdot \xi=\xi'$. Assume that $\Gamma$ has a cusp at $\infty$, define the lattice
$$\mathcal{O}_{\Gamma}:=\{\textbf{x}\in \mathbb{V}^{n-1}\ |\ u_{\textbf{x}}\in \Gamma_{\infty}'\}$$
in $\mathbb{V}^{n-1}$. Let $\mathcal{F}_{\mathcal{O}_{\Gamma}}\subset \mathbb{V}^{n-1}$ be a fundamental domain for $\mathcal{O}_{\Gamma}$. One easily sees that the set
\begin{equation}
\label{fu} \mathcal{F}_{\infty}'= \{u_{\textbf{x}}a_tk\ |\ \textbf{x}\in \mathcal{F}_{\mathcal{O}_{\Gamma}}, t\in \mathbb{R}, k\in K\}
\end{equation}
is a fundamental domain for $\Gamma_{\infty}'\backslash G$. It contains $[\Gamma_{\infty}: \Gamma_{\infty}']$ copies of $\Gamma_{\infty}\backslash G$. For later use, we note that since $\Gamma_{\infty}'$ is a lattice in $N$, $\Gamma_{\infty}\backslash Q$ is relatively compact, hence $$\omega_{\Gamma}:=\int_{\Gamma_{\infty}\backslash Q}dq$$
is finite.

For any $\tau\in \bbR$, let us denote $A(\tau)= \{a_t\in A\ |\ t\geq \tau\}$. Recall that a Siegel set is a subset of $G$ of the form $\Omega_{\tau,U}=U A(\tau)K$ where $U$ is an open, relatively compact subset of $N$. Since $G$ is of real rank one, we can apply the reduction theory of Garland and Raghunathan (\cite{Ga} Theorem 0.6). That is, there exists $\tau_0\in \bbR$, an open, relatively compact subset $U_0\subset N$, a finite set $\Xi=\{\xi_1,\cdots,\xi_h\}\subset G$ (corresponding to a complete set of $\Gamma$-inequivalent cusps) and an open, relatively compact subset $\mathfrak{C}$ of $G$ such that the Siegel fundamental domain
\begin{equation}
\label{sf} \mathcal{F}_{\Gamma, \tau_0,U_0}= \mathfrak{C}\bigcup\left(\bigcup_{\xi_j\in \Xi}\xi_j\Omega_{\tau_0,U_0}\right),
\end{equation} satisfies the following properties:
\begin{enumerate}
\item[(1)] $\Gamma\mathcal{F}_{\Gamma, \tau_0, U_0}= G$;
\item[(2)] the set $\{\gamma\in \Gamma\ |\ \gamma\mathcal{F}_{\Gamma, \tau_0, U_0}\cap \mathcal{F}_{\Gamma, \tau_0, U_0}\neq \emptyset\}$ is finite;
\item[(3)] $\gamma \xi_i\Omega_{\tau_0,U_0}\cap \xi_j\Omega_{\tau_0,U_0}= \emptyset$ for all $\gamma\in \Gamma$ whenever $\xi_i\neq \xi_j\in\Xi$.
\end{enumerate}

In other words, the restriction to $\mathcal{F}_{\Gamma, \tau_0,U_0}$ of the natural projection of $G$ onto $\Gamma\backslash G$ is surjective, at most finite-to-one and the cusp neighborhood of each cusp of $\Gamma\backslash G$ can be taken to be disjoint. We will fix this Siegel fundamental domain $\mathcal{F}_{\Gamma, \tau_0,U_0}$ throughout the paper. For further use, we note that $U_0$ contains a fundamental domain of $\Gamma_{\infty}'\backslash N$.

\subsection{The distance function.}
Fix a non-uniform lattice $\Gamma$ in $G$. Let $\textrm{dist}_G$ and $\textrm{dist}=\textrm{dist}_{\Gamma}$ denote the hyperbolic distance functions on $G/K=\mathbb{H}^{n+1}$ and $\Gamma\backslash G/K= \Gamma\backslash \mathbb{H}^{n+1}$ respectively. By slightly abuse of notation,  we also denote $\textrm{dist}_{G}$ and $\textrm{dist}$ to be their lifts to $G$ and $\Gamma\backslash G$ respectively. In particular, $\textrm{dist}_G$ is left $G$-invariant and satisfies $\textrm{dist}_G\left(I_2, a_tk\right)= t$ for any $t\geq 0$ and $k\in K$, where $I_2$ is the identity matrix in $G$. Moreover, for any $g,h\in G$, $\textrm{dist}_G$ and $\textrm{dist}$ satisfy the relation
$$\textrm{dist}\left(\Gamma g, \Gamma h\right)=\inf_{\gamma\in \Gamma}\textrm{dist}_G\left(g,\gamma h\right).$$

Clearly, $\textrm{dist}\left(\Gamma g,\Gamma h\right)\leq \textrm{dist}_{G}\left(g, h\right)$. Conversely, if $g, h$ are from the Siegel set $\Omega_{\tau_0, U_0}$, then there exists a constant $D$ such that $\textrm{dist}_{G}\left(\xi_ig, \gamma\xi_j h\right)\geq \textrm{dist}_{G}\left(g, h\right)- D$ for any $\xi_i,\xi_j\in\Xi$ and any $\gamma\in \Gamma$ (see \cite[Theorem C]{Bo}). In particular, this implies
$$\textrm{dist}\left(\Gamma\xi_jg, \Gamma\xi_j h\right)\geq \textrm{dist}_{G}\left(g, h\right)- D$$
for any $\xi_j\in \Xi$ and any $g, h\in\Omega_{\tau_0, U_0}$. We then have
\begin{Lem}
\label{Lem.dist} For $o\in \mathcal{F}_{\Gamma, \tau_0, U_0}$ fixed, there exists a constant $D'$ such that
\begin{equation}
 \label{dist}\textrm{dist}_{G}\left(o, \xi_j g\right)- D'\leq \textrm{dist}\left(o, \xi_j g\right)\leq \textrm{dist}_{G}\left(o, \xi_j g\right)
\end{equation}
for any $\xi_j\in \Xi$ and any $g\in\Omega_{\tau_0, U_0}$.
\end{Lem}
\begin{Rem}
We view $o, \xi_jg$ as elements in $\Gamma\backslash G$ when we write $\textrm{dist}\left(o,\xi_j g\right)$, and as elements in $G$ when we write $\textrm{dist}_{G}\left(o, \xi_j g\right)$.
\end{Rem}
\begin{proof}
Only the first inequality needs a proof. Fix an arbitrary $h\in\Omega_{\tau_0, U_0}$, we have
\begin{align*}
\textrm{dist}\left(o, \xi_jg\right)&\geq \textrm{dist}\left(\xi_jh, \xi_jg\right)-\textrm{dist}\left(o, \xi_jh\right)\\
                   &\geq \textrm{dist}_{G}\left(h, g\right)-D-\textrm{dist}_{G}\left(o, \xi_jh\right)\\
                   &= \textrm{dist}_{G}\left(\xi_jh, \xi_jg\right)-D-\textrm{dist}_{G}\left(o, \xi_jh\right)\\
                   &\geq \textrm{dist}_{G}\left(o, \xi_jg\right)-2\textrm{dist}_{G}\left(o, \xi_jh\right)-D.
\end{align*}
Then $D'= 2\sup_{\xi_j\in \Xi}\textrm{dist}_{G}\left(o, \xi_jh\right)+ D$ satisfies $\left(\ref{dist}\right)$.
\end{proof}
Note that any $g\in \Omega_{\tau_0,U_0}$ can be written as $g=ua_tk$ with $u\in U_0, t\geq \tau_0, k\in K$. Since $U_0$ is relatively compact, $\Xi$ is finite and $\textrm{dist}$ is right $K$-invariant, in view of Lemma \ref{Lem.dist} we have
\begin{equation}
\label{final} \textrm{dist}\left(o, \xi_jg\right)= \textrm{dist}_{G}\left(o, a_t\right)+ O\left(1\right)= t+ O\left(1\right).
\end{equation}
In particular, if $\Gamma$ has a cusp at $\infty$, $\Xi$ can be taken such that it contains the identity element. Hence in this case,
\begin{equation}
\label{final 2}\textrm{dist}\left(o, g\right)= t+ O(1)
\end{equation}
for any $g=ua_tk\in \Omega_{\tau_0, U_0}$.
Finally, we note that when $r$ is sufficiently large, $B_r$ is a collection of neighborhoods at all cusps. In view of the above reduction theory and the Haar measure $\left(\ref{haar}\right)$ we have
\begin{equation}
\label{norme}\sigma\left(B_r\right)\asymp e^{-nr}
\end{equation}
for any $r>0$.

\section{Incomplete Eisenstein series}
Let $G$ be as before, $\Gamma\subset G$ be a non-uniform lattice in $G$ with a cusp at $\infty$. Given a compactly supported function $f\in L^2(Q\backslash G)$, we define the incomplete Eisenstein series attached to $f$ by
$$\Theta_f(g)= \sum_{\gamma\in \Gamma_{\infty}\backslash \Gamma}f\left(\gamma g\right).$$
Note that $\Theta_f$ is left $\Gamma$-invariant since $f$ is left $\Gamma_{\infty}$-invariant ($\Gamma_{\infty}\subset Q$). Moreover, since $f$ is compactly supported, the summation is a finite sum. Hence it is a well-defined function on $\Gamma\backslash G$. We first give a simple but useful identity related to $\Theta_f$ given by the standard unfolding trick.
\begin{Lem}
\label{unfold} For $\Theta_f$ as above and any $F\in L^2\left(\Gamma\backslash G\right)$
$$\int_{\Gamma\backslash G}\Theta_f(g)F\left(g\right)d\sigma\left(g\right)= \int_{\Gamma_{\infty}\backslash G}f\left(g\right)F\left(g\right)d\sigma\left(g\right).$$
\end{Lem}
\begin{proof}
Let $\mathcal{F}_{\Gamma}$ be a fundamental domain for $\Gamma\backslash G$. Note that $\mathcal{F}_{\infty}= \cup_{\gamma\in \Gamma_{\infty}\backslash \Gamma}\gamma \mathcal{F}_{\Gamma}$ form a fundamental domain for $\Gamma_{\infty}\backslash G$, hence
\begin{align*}
\int_{\mathcal{F}_{\Gamma}}\Theta_f(g)F\left(g\right)d\sigma\left(g\right)&= \sum_{\gamma\in \Gamma_{\infty}\backslash \Gamma}\int_{\mathcal{F}_{\Gamma}}f\left(\gamma g\right)F\left(g\right)d\sigma\left(g\right)\\
                                                    &= \sum_{\gamma\in \Gamma_{\infty}\backslash \Gamma}\int_{\gamma\mathcal{F}_{\Gamma}}f\left(g\right)F\left(g\right)d\sigma\left(g\right)\\
                                                    &= \int_{\bigcup_{\gamma\in \Gamma_{\infty}\backslash \Gamma}\gamma \mathcal{F}_{\Gamma}}f\left(g\right)F\left(g\right)d\sigma\left(g\right)\\
                                                    &= \int_{\Gamma_{\infty}\backslash G}f\left(g\right)F\left(g\right)d\sigma\left(g\right).\qedhere
\end{align*}
\end{proof}
In particular, taking $F= \overline{\Theta_f}= \Theta_{\overline{f}}$ we get
\begin{equation}
 ||\Theta_f||_2^2= \int_{\Gamma_{\infty}\backslash G}\overline{f\left(g\right)}\Theta_f\left(g\right)d\sigma\left(g\right).
\end{equation}
Moreover, by $\left(\ref{haar}\right)$ and $\left(\ref{fu}\right)$ we have
\begin{equation}
\label{L2} ||\Theta_f||_2^2= \frac{1}{[\Gamma_{\infty}: \Gamma_{\infty}']\nu_{\Gamma}}\int_{K}\int_{\bbR}\overline{f\left(a_tk\right)}e^{-nt}\int_{\mathcal{F}_{\mathcal{O}_{\Gamma}}}\Theta_f\left(u_{\textbf{x}}a_tk\right)d\textbf{x}dtdk.
\end{equation}
Next we will compute $||\Theta_f||_2^2$ by expressing the term $\int_{\mathcal{F}_{\mathcal{O}_{\Gamma}}}\Theta_f\left(u_{\textbf{x}}a_tk\right)d\textbf{x}$ as an integral of certain non-spherical Eisenstein series (see Lemma $\ref{l.bound}$). Before we can do that, we need to recall some facts of spherical Eisenstein series of real rank one groups (see \cite{GW} for the statements and \cite{RL} for the general theory).

\subsection{Spherical Eisenstein series.}
Denote by $C^{\infty}\left(Q\backslash G/K\right)$ the space of smooth left $Q$-invariant and right $K$-invariant functions on $G$. For any $s\in \bbC$, define the function $\varphi_s\in C^{\infty}\left(Q\backslash G/K\right)$ by
\begin{equation}
\label{eigen} \varphi_s\left(ua_tk\right)= e^{st}.
\end{equation}

Given a lattice $\Gamma\subset G$ with a cusp at $\infty$, the spherical Eisenstein series (corresponding to the cusp at $\infty$) is defined by $$E\left(s, g\right)= \sum_{\gamma\in \Gamma_{\infty}\backslash \Gamma}\varphi_s\left(\gamma g\right).$$ This series converges for $\textrm{Re}\left(s\right)> n$, it is right $K$-invariant and satisfies the following differential equation
$$\left(\Delta+s\left(n-s\right)\right)E\left(s,g\right)=0,$$
where $\Delta=e^{2t}\left(\frac{\partial^2}{\partial x_0^2}+\cdots+\frac{\partial^2}{\partial x_{n-1}^2}\right)+\left(\frac{\partial^2}{\partial t^2}-n\frac{\partial}{\partial t}\right)$ is the Laplace-Beltrami operator on the upper half space
$$\mathbb{H}^{n+1}= \left\{x_0+ x_1e_1+ \cdots+ x_{n-1}e_{n-1}+ e^{t}e_n\ \big|\ x_i, t\in \mathbb{R}\right\}.$$
The constant term of the Eisenstein series (corresponding to the cusp at $\infty$) is defined as
$$E^0\left(s, g\right)= \frac{1}{|\mathcal{F}_{\mathcal{O}_{\Gamma}}|}\int_{\textbf{x}\in \mathcal{F}_{\mathcal{O}_{\Gamma}}}E\left(s,u_{\textbf{x}}g\right)d\textbf{x}.$$ It has the form
$$E^0(s, g)= \varphi_s(g)+ \mathcal{C}_{\Gamma}(s)\varphi_{n-s}(g),$$ where the function $\mathcal{C}(s)=\mathcal{C}_{\Gamma}(s)$ can be extended to a meromorphic function on the half plane $\textrm{Re}(s)\geq \frac{n}{2}$ with a simple pole at $s=n$ and only possibly finitely many simple poles (called exceptional poles) on the interval $(\frac{n}{2}, n)$. Finally we note that
using the functional equation satisfied by the Eisenstein series and the same argument in \cite[p. 10]{KM}, $|\mathcal{C}(s)|\leq 1$ for $\textrm{Re}(s)= \frac{n}{2}$.

\subsection{The raising operator}
Let $\mathfrak{g}$ and $\mathfrak{k}$ be the Lie algebra of $G$ and $K$ respectively. Let $\fg_{\bbC}=\fg\otimes_{\bbR}\bbC$ and $\fk_{\bbC}=\fk\otimes_{\bbR}\bbC$ be their complexifications. As a real vector space, $\mathfrak{k}$ is spanned by the matrices
$$ -\frac12\begin{pmatrix}
e_ie_j & 0\\
0 & e_ie_j\end{pmatrix} \left(1\leq i<j\leq n-1\right),\quad -\frac12\begin{pmatrix}e_l &0\\
0& -e_l\end{pmatrix} \left(1\leq l\leq n-1\right),$$
$$\frac12\begin{pmatrix} 0& e_m\\
e_m &0\end{pmatrix} \left(1\leq m\leq n-1\right),\quad \frac12\begin{pmatrix}
0 & 1\\
-1 & 0\end{pmatrix},$$
where $e_1,\cdots, e_{n-1}$ are elements in the Clifford algebra $C_n$ as before.
The Lie algebra $\mathfrak{g}$ is spanned by the matrices as above and
$$\begin{pmatrix} 0 & e_i\\
0 & 0\end{pmatrix} \left(1\leq i\leq n-1\right),\quad \begin{pmatrix} 0 & 1\\
0 & 0\end{pmatrix},\quad \begin{pmatrix}
1& 0\\
0& -1\end{pmatrix}.$$

\subsubsection{Root-space decomposition of $\fk_{\bbC}$}
Let $\mathfrak{h}$ be a Cartan subalgebra of $\mathfrak{k}_{\bbC}$. Since $\mathfrak{k}_{\bbC}$ is a complex semisimple Lie algebra, it has a root-space decomposition with respect to $\mathfrak{h}$:
$$\mathfrak{k}_{\bbC}=\mathfrak{h}\oplus\bigoplus_{\alpha\in \Phi\left(\mathfrak{k}_{\bbC},\mathfrak{h}\right)}\mathfrak{k}_{\alpha},$$
where $\Phi=\Phi\left(\mathfrak{k}_{\bbC},\mathfrak{h}\right)$ is the corresponding set of roots, and for each $\alpha\in \Phi$ the root-space $\mathfrak{k}_{\alpha}$ is given by
$$\mathfrak{k}_{\alpha}:=\left\{X\in \mathfrak{k}_{\bbC}\ |\ [H, X]=\alpha\left(H\right)X\ \textrm{for any $H\in\fh$}\right\}.$$
Each root-space is one-dimensional and satisfies $[\mathfrak{k}_{\alpha},\mathfrak{k}_{\beta}]\subset \mathfrak{k}_{\alpha+\beta}$ for any $\alpha,\beta\in \Phi$. Fix a set of simple roots $\Delta$ and let $\Phi^{+}$  denote the corresponding set of positive roots. Then $\Phi=\Phi^{+}\cup \left(-\Phi^{+}\right)$. For backgrounds on complex semisimple Lie algebra, see \cite[Chaper II]{Ant}. In this subsection, we first give an explicit isomorphism between $K$ and $SO\left(n+1\right)$, then use this isomorphism and the classical root-space decomposition of $\mathfrak{so}\left(n+1,\bbC\right)$ to get an explicit root-space decomposition of $\fk_{\bbC}$.

Recall the identification
\begin{align*}
M\backslash K &\longrightarrow S^n:=\left\{x_0+x_1e_1+\cdots +x_ne_n\ |\ x_i\in\bbR, x_0^2+x_1^2+\cdots +x_n^2=1\right\}\\
\begin{pmatrix}
q_2' & -q_1'\\
q_1 & q_2\end{pmatrix}&\longmapsto \qquad\qquad\qquad 2\overline{q_1}q_2+ \left(|q_2|^2-|q_1|^2\right)e_n.
\end{align*}
Embed $S^n$ in $\mathbb{V}^n$ and fix an inner product on $\mathbb{V}^{n}$ such that $\{1,e_1,\cdots,e_n\}$ form a orthonormal basis of $\mathbb{V}^{n}$. Then the right regular action of $K$ on $M\backslash K=S^n$ induces an isomorphism from $K$ to $SO\left(n+1\right)$. In particular, it induces an isomorphism between $\fk_{\bbC}$ and $\mathfrak{so}\left(n+1,\bbC\right)$. Explicitly, for any $0\leq i<j\leq n$ define $L_{i,j}\in \mathfrak{k}_{\bbC}$ as following:
\begin{equation}
\label{laiso} L_{i,j}= \left\{
  \begin{array}{lr}
   -\frac12\begin{pmatrix}
e_ie_j & 0\\
0 & e_ie_j\end{pmatrix} & \textrm{if}  \ \ 1\leq i< j\leq n-1\\
   -\frac12\begin{pmatrix}e_j &0\\
0& -e_j\end{pmatrix} & \ \  \textrm{if} \ \ i=0, 1\leq j\leq n-1\\
  \frac12\begin{pmatrix} 0& e_i\\
e_i &0\end{pmatrix} & \textrm{if}  \ \ 1\leq i\leq n-1, j=n\\
  \frac12\begin{pmatrix}
0 & 1\\
-1 & 0\end{pmatrix} & \textrm{if} \ \ i=0, j=n.
  \end{array}
\right.
\end{equation}
By direct computation, the induced isomorphism from $\fk_{\bbC}$ to $\mathfrak{so}\left(n+1,\bbC\right)$ is given by sending $L_{i,j}$ to $E_{i,j}$ for any $0\leq i<j\leq n$, where $E_{i,j}$ is the antisymmetric $\left(n+1\right)\times \left(n+1\right)$ matrix with $\left(i,j\right)^{\textrm{th}}$ entry equals one, $\left(j,i\right)^{\textrm{th}}$ entry equals negative one and zero elsewhere. Using the classical commutator relations of $E_{i,j}$ we get the commutator relations of $L_{i,j}$. To ease the notation, we let $L_{i,i}=0$ for $0\leq i\leq n$ and  $L_{i,j}=-L_{j,i}$ for $0\leq j< i\leq n$. Explicitly, $L_{i,j}$ satisfy the following commutator relations
\begin{equation}
\label{comr} [L_{i,j}, L_{l,m}]=\delta_{jl}L_{i,m}-\delta_{il}L_{j,m}-\delta_{jm}L_{il}+ \delta_{im}L_{j,l}
\end{equation}
for any $0\leq i,j,l,m\leq n$, where $\delta_{ij}$ is the Kronecker symbol.
Moreover, using the root-space decomposition of
$\mathfrak{so}\left(n+1,\bbC\right)$ (see \cite[p. 127-129]{Ant}) we get the following root-space decomposition of $\mathfrak{k}_{\bbC}$ depending on the parity of $n+1$.

\subsubsection*{\textbf{Case I: $n+1= 2k+1$ is odd}}
For each $0\leq i\leq k-1$, let
$$H_i= \sqrt{-1}L_{2i, 2i+1}$$
with $L_{2i,2i+1}$ defined in $\left(\ref{laiso}\right)$. Let $\mathfrak{h}$ be the complex vector space spanned by the set $\left\{H_i\ |\ 0\leq i\leq k-1\right\}$. For each $0\leq i\leq k-1$, let $\varepsilon_i: \mathfrak{h}\to \bbC$ be the linear functional on $\mathfrak{h}$ characterised by $\varepsilon_{i}\left(H_j\right)=\delta_{ij}$. Using the above isomorphism between $\mathfrak{k}_{\bbC}$ and the root-space decomposition of $\mathfrak{so}\left(2k+1,\bbC\right)$, we know $\fh$ is a Cartan subalgebra and we can choose the set of simple roots to be
$$\Delta=\left\{\varepsilon_0-\varepsilon_1, \varepsilon_1-\varepsilon_2,\ldots,\varepsilon_{k-2}-\varepsilon_{k-1},\varepsilon_{k-1}\right\}.$$
The corresponding positive roots are given by
$$\Phi^{+}= \left\{\varepsilon_{i}\pm\varepsilon_{j}\ |\ 0\leq i< j\leq k-1\right\}\cup \left\{\varepsilon_l\ |\ 0\leq l\leq k-1\right\}.$$
Moreover, for any $0\leq i<j\leq k-1$ and $0\leq l\leq k-1$, the positive root-spaces $\fk_{\varepsilon_i\pm\varepsilon_j}$ and $\fk_{\varepsilon_l}$ are given as following:
\begin{equation}
\label{po}\fk_{\varepsilon_i\pm\varepsilon_j}=\bbC\left\langle \left(L_{2i,2j}-\sqrt{-1}L_{2i+1,2j}\right)\mp\left(L_{2i+1,2j+1}+\sqrt{-1}L_{2i,2j+1}\right)\right\rangle,
\end{equation}
and
\begin{equation}
\label{po2}\fk_{\varepsilon_l}=\bbC\left\langle L_{2l,2k}-\sqrt{-1}L_{2l+1,2k}\right\rangle.
\end{equation}

\subsubsection*{\textbf{Case II: $n+1=2k$ is even}} Similar to the odd case, for each $0\leq i\leq k-1$, let
$$H_i= \sqrt{-1}L_{2i, 2i+1}$$
and let $\mathfrak{h}$ be the complex vector space spanned by $\left\{H_i\ |\ 0\leq i\leq k-1\right\}$. For each $0\leq i\leq k-1$, denote $\varepsilon_i: \mathfrak{h}\to \bbC$ the linear functional on $\mathfrak{h}$ characterised by $\varepsilon_{i}\left(H_j\right)=\delta_{ij}$. The set of simple roots can be chosen to be
$$\Delta=\left\{\varepsilon_0-\varepsilon_1,\varepsilon_1-\varepsilon_2,\ldots, \varepsilon_{k-2}-\varepsilon_{k-1}, \varepsilon_{k-2}+\varepsilon_{k-1}\right\}.$$
The corresponding positive roots are given by
$$\Phi^{+}=\left\{\varepsilon_i\pm\varepsilon_j\ |\ 0\leq i<j\leq k-1\right\},$$
with $\fk_{\varepsilon_i\pm\varepsilon_j}$ also given by $\left(\ref{po}\right)$.
\begin{Rem}
The commutator relations and root-space decomposition above can both be checked directly using the relations $e_ie_j+e_je_i= -2\delta_{i,j}$ for any $1\leq i,j\leq n$.
\end{Rem}
\subsubsection{Spherical principal series representation}
Let $G=NAK$ be the fixed Iwasawa decomposition and $M$ be the centralizer of $A$ in $K$ as before. For any $s\in\bbC$, recall the function $\varphi_s$ on $NA$ defined by
$$\varphi_s\left(ua_t\right)= e^{st}$$
for any $u\in N$ and $a_t\in A$. Consider the corresponding spherical principal series representation $I^s=\textrm{Ind}_{NAM}^{G}\left(\varphi_s\otimes 1_{M}\right)$, where $1_{M}$ is the trivial representation of $M$. Elements in $I^s$ are measurable functions $f: G\to \bbC$ satisfying
\begin{equation}
\label{prser} f\left(uawg\right)=\varphi_s\left(a\right)f\left(g\right)\ \textrm{for $\sigma$-a.e. $g\in G$, with $u\in N, a\in A$ and $w\in M$}.
\end{equation}
$G$ acts on $I^s$ by right regular action. We note that due to condition $\left(\ref{prser}\right)$, $f\in I^s$ is a function on $A\times M\backslash K$, thus by the identification between $M\backslash K$ and $S^n$, $f$ is a functions in coordinates $\left(t,x_0,x_1,\cdots,x_n\right)$ with the restriction $x_0^2+x_1^2\cdots+x_n^2=1$.

Let $I^s_{\infty}$ be the space of smooth functions in $I^s$. We note that $I^s_{\infty}$ is a dense subspace of $I^s$ and the right regular action of $G$ on $I^s$ induces a $\fg$-module structure on $I^s_{\infty}$: For any $X\in \fg$ and any $f\in I^s_{\infty}$, define the Lie derivative, $\pi\left(X\right)$, by
$$\left(\pi\left(X\right)f\right)\left(g\right)= \frac{d}{dy}f\left(g\exp\left(yX\right)\right)\bigg|_{y=0}.$$
We note that the Lie derivative respects the Lie bracket, that is, $[\pi\left(X\right),\pi\left(Y\right)]=\pi\left([X,Y]\right)$ for any $X,Y\in \mathfrak{g}$, where the first Lie bracket is the Lie bracket of endomorphisms. Since functions in $I^s_{\infty}$ are complex-valued, we can complexify the Lie derivative by defining $$\pi\left(X+\sqrt{-1}Y\right):=\pi\left(X\right)+\sqrt{-1}\pi\left(Y\right)$$ for any $X,Y\in\fg$. Thus $I^s_{\infty}$ becomes a $\fg_{\bbC}$-module. In particular, $I^s_{\infty}$ is also a $\fk_{\bbC}$-module. Let $\fh$ and $\Phi^{+}$ be as before. Let $\fh^{\ast}$ denote the complex dual of $\fh$. Given a $\fk_{\bbC}$-module $V$ and $\rho\in \fh^{\ast}$, we say $v\in V$ is of $K$-weight $\rho$ if $H\cdot v=\rho\left(H\right)v$ for any $H\in\fh$. We say $v\in V$ is a highest weight vector if $v$ is of $K$-weight $\rho$ for some $\rho\in \fh^{\ast}$ and $X\cdot v=0$ for any $\alpha\in \Phi^+$ and any $X\in \fk_{\alpha}$. We note that every irreducible representation of $K$ is a finite-dimensional irreducible $\fk_{\bbC}$-module by differentiating the group action at the identity, and every finite-dimensional irreducible $\fk_{\bbC}$-module admits a unique (up to scalars) highest weight vector (see \cite[Theorem 5.5 (b)]{Ant}).

Due to condition $\left(\ref{prser}\right)$, $I^s$ is isomorphic to $L^2\left(M\backslash K\right)$ as $K$-representations by sending $f$ to $f|_K$. Identify $M\backslash K$ with $S^n$ as above, we have the following decomposition of $L^2\left(M\backslash K\right)$ as $K$-representations:
$$L^2\left(M\backslash K\right)=\hat{\bigoplus_{m\geq0}}L^2\left(M\backslash K, m\right),$$
where $L^2\left(M\backslash K, m\right)$ is the space of degree $m$ harmonic polynomials in $n+1$ variables restricted to $S^n$ (see \cite[Corollary 5.0.3]{Pau}) and $\hat{\bigoplus}$ denote the Hilbert direct sum.  Moreover, let $\mathcal{H}^m$ be the space of degree $m$ harmonic polynomials in coordinates $\left(x_0,x_1,\cdots, x_n\right)\in\mathbb{V}^{n}$. Then $\mathcal{H}^m$ is an irreducible $K$-representation and is isomorphic to $L^2\left(M\backslash K, m\right)$ via the map $\phi\mapsto\phi|_{S^n}$ (\cite[Theorem 0.3 and 0.4]{TTT}). Finally, we note that $\left(x_0-\sqrt{-1}x_1\right)^m\in\mathcal{H}^m$ is of $K$-weight $m\varepsilon_0$ (\cite[p. 277-278]{Ant}) and $\mathcal{H}^m$ is of highest weight $m\varepsilon_0$ (\cite[p. 339 Problem 9.2]{Ant}). Hence $\left(x_0-\sqrt{-1}x_1\right)^m$ is the unique (up to scalars) highest weight vector in $\mathcal{H}^m$.

Correspondingly, let $I_{\infty}^s\left(K,m\right):=\left\{f\in I_{\infty}^s\ |\ f|_K\in L^2\left(M\backslash K,m\right)\right\}$. Then we have a decomposition of $I^s_{\infty}$
$$I^s_{\infty}=\bigoplus_{m=0}^{\infty}I_{\infty}^s\left(K,m\right).$$ Moreover, $I_{\infty}^s\left(K,m\right)$ is an irreducible $\fk_{\bbC}$-module of highest weight $m\varepsilon_0$, and the highest weight vector is given by
$$\varphi_{s,m}\left(t,x_0, x_1,\ldots,x_n\right):= e^{st}\left(x_0-\sqrt{-1}x_1\right)^m.$$

Now we define the raising operator $R^{+}\in \fg_{\bbC}$ by
\begin{align*}
R^{+}&=\frac12\pi\left(\begin{pmatrix}0 & -1+\sqrt{-1}e_1\\
-1-\sqrt{-1}e_1 & 0\end{pmatrix}\right)\\
&= -\frac12\pi\left(\begin{pmatrix}
0 & 1\\
1 & 0\end{pmatrix}\right)+\frac{\sqrt{-1}}{2}\pi\left(\begin{pmatrix}
0 & e_1\\
-e_1 & 0\end{pmatrix}\right).
\end{align*}
To compute $R^{+}$ explicitly, we use the spherical coordinates on $S^n$: Let $\left(x_0,x_1,\ldots,x_n\right)$ be the coordinates on $S^n$ as above, define $\left(\theta_0,\theta_1,\ldots,\theta_{n-1}\right)\in [0,2\pi]^{n-1}\times [0,\pi)$ such that
\begin{align*}
x_0&=\cos\theta_0,\\
x_1&=\sin\theta_0\cos\theta_1,\\
 &\ \ \vdots \\
x_{n-1}&=\sin\theta_0\cdots\sin\theta_{n-2}\cos\theta_{n-1}, \\ x_n&=\sin\theta_0\cdots\sin\theta_{n-2}\sin\theta_{n-1}.
\end{align*}
Hence under the coordinates $\left(t,\theta_i\right)$, $\varphi_{s,m}$ is given by
$$\varphi_{s,m}\left(t,\theta_i\right)= e^{st}\left(\cos\theta_0-\sqrt{-1}\sin\theta_0\cos\theta_1\right)^m.$$
Moreover, in these coordinates, for any $X\in \fg_{\bbC}$, the Lie derivative $\pi\left(X\right)$ is a first order differential operator of the form
$$\pi\left(X\right)= F\frac{\partial}{\partial t}+\sum_{i=0}^{n-1}F_i\frac{\partial}{\partial \theta_i},$$
where $F, F_i$ are functions in $\left(t,\theta_i\right)$. For our purpose, we define $$\widetilde{\pi\left(X\right)}:=F\frac{\partial}{\partial t}+ F_0\frac{\partial}{\partial \theta_0}+ F_1\frac{\partial}{\partial \theta_1}.$$
Since $\varphi_{s,m}$ only depends on the variables $\left(t,\theta_0,\theta_1\right)$, $\pi\left(X\right)\varphi_{s,m}=\widetilde{\pi\left(X\right)}\varphi_{s,m}$
for any $X\in \fg_{\bbC}$.

Now we describe the strategy to compute the Lie derivatives. We first show how to extract the coordinates $\left(t,\theta_i\right)$ from a given element $g=\begin{pmatrix}
a & b\\
c& d\end{pmatrix}\in G$. Write
\begin{equation}
\label{coord}\begin{pmatrix}
a & b\\
c &d\end{pmatrix}=\begin{pmatrix}
1 & u\\
0 & 1\end{pmatrix}\begin{pmatrix}
e^{\frac{t}{2}} & 0\\
0 & e^{-\frac{t}{2}}\end{pmatrix}\begin{pmatrix}
q_2' & -q_1' \\
q_1 & q_2\end{pmatrix}
\end{equation}
by Iwasawa decomposition. Comparing the second row of the matrices on both sides, we get
\begin{equation}
\label{extract 1}e^{-t}=|c|^2+|d|^2
\end{equation}
and
\begin{equation}
\label{extract 2}x_0+x_1e_1+\cdots+x_ne_n=\frac{2\bar{c}d+\left(|d|^2-|c|^2\right)e_n}{|c|^2+|d|^2},
\end{equation}
where $x_i$ are expressed by $\theta_i$ as above. Fix an element $g\in G$. For any $X\in\fg$, the coordinates $\left(t,\theta_i\right)$ of $g\exp\left(yX\right)$ can be viewed as functions in $y$ as $y$ varies. Denote $\left(t\left(y\right),\theta\left(y\right)\right)$ to indicate this dependence on $y$. Then the Lie derivative $\widetilde{\pi\left(X\right)}$ is exactly given by $$\widetilde{\pi\left(X\right)}=t'\left(0\right)\frac{\partial}{\partial t}+ \theta'_0\left(0\right)\frac{\partial}{\partial \theta_0}+ \theta_1'\left(0\right)\frac{\partial}{\partial \theta_1}.$$

\begin{Lem}
\label{ecompu}Let $B_1=\begin{pmatrix}
0 & 1\\
1 & 0\end{pmatrix}$ and $B_2=\begin{pmatrix}
0 & e_1\\
-e_1 & 0\end{pmatrix}$. Then
$$\widetilde{\pi\left(B_1\right)}=-2\cos\theta_0\frac{\partial}{\partial t}-2\sin\theta_0\frac{\partial}{\partial \theta_0}$$
and
$$\widetilde{\pi\left(B_2\right)}=-2\sin\theta_0\cos\theta_1\frac{\partial}{\partial t}+2\cos\theta_0\cos\theta_1\frac{\partial}{\partial \theta_0}-2\frac{\sin\theta_1}{\sin\theta_0}\frac{\partial}{\partial \theta_1}.$$
\end{Lem}
\begin{proof}
Using the formula $\exp\left(yB_1\right)=\sum_{i=0}^{\infty}\frac{\left(yB_1\right)^i}{i!}$ we get $$\exp\left(yB_1\right)=\begin{pmatrix}
\cosh y & \sinh y\\
\sinh y & \cosh y\end{pmatrix}.$$
Thus
$$\begin{pmatrix}
a & b\\
c &d\end{pmatrix}\begin{pmatrix}
\cosh y & \sinh y\\
\sinh y & \cosh y\end{pmatrix}=\begin{pmatrix}
\star &\star \\
c\cosh y +d\sinh y & c\sinh y + d\cosh y \end{pmatrix}.$$
Using $\left(\ref{extract 1}\right)$ we get
$$e^{-t\left(y\right)}=|\cosh y +d\sinh y|^2+|c\sinh y + d\cosh y|^2=e^{-t}\left(\cosh\left(2y\right)+\sinh\left(2y\right)\cos\theta_0\right).$$
Taking derivatives with respect to $y$ and evaluating at $0$ on both sides we get
$$t'\left(0\right)=-2\cos\theta_0.$$
Similarly, using $\left(\ref{extract 2}\right)$ and comparing the constant term and coefficient of $e_1$, we get
$$\cos\left(\theta_0\left(y\right)\right)= \frac{\sinh\left(2y\right)+\cosh\left(2y\right)\cos\theta_0}{\cosh\left(2y\right)+\sinh\left(2y\right)\cos\theta_0}$$
and
$$\sin\left(\theta_0\left(y\right)\right)\cos\left(\theta_1\left(y\right)\right)=\frac{\sin\theta_0\cos\theta_1}{\cosh\left(2y\right)+\sinh\left(2y\right)\cos\theta_0}.$$
Taking derivatives with respect to $y$ and evaluating at $0$ we get
$$\theta_0'\left(0\right)= -2\sin\theta_0\quad\textrm{and}\quad \theta_1'\left(0\right)= 0.$$
Thus $\widetilde{\pi\left(B_1\right)}=-2\cos\theta_0\frac{\partial}{\partial t}-2\sin\theta_0\frac{\partial}{\partial \theta_0}$.

Similarly, for $B_2$ we have $\exp\left(yB_2\right)=\begin{pmatrix}
\cosh y & \sinh y e_1\\
-\sinh y e_1& \cosh y\end{pmatrix}$ and
$$\begin{pmatrix}
a & b\\
c &d\end{pmatrix}\begin{pmatrix}
\cosh y & \sinh y e_1\\
-\sinh y e_1& \cosh y\end{pmatrix}=\begin{pmatrix}
\star &\star \\
c\cosh y-de_1\sinh y& ce_1\sinh y+ d\cosh y\end{pmatrix}.$$
Using $\left(\ref{extract 1}\right)$ and $\left(\ref{extract 2}\right)$, after some tedious but straightforward computations we get
\begin{align*}
e^{-t\left(y\right)}&= e^{-t}\left(\cosh\left(2y\right)+\sinh\left(2y\right)\sin\theta_0\cos\theta_1\right),\\
\cos\left(\theta_0\left(y\right)\right)&=\frac{\cos\theta_0}{\cosh\left(2y\right)+\sinh\left(2y\right)\sin\theta_0\cos\theta_1},
\end{align*}
and
$$\sin\left(\theta_0\left(y\right)\right)\cos\left(\theta_1\left(y\right)\right)=\frac{\sinh\left(2y\right)+\cosh\left(2y\right)\sin\theta_0\cos\theta_1}{\cosh\left(2y\right)+\sinh\left(2y\right)\sin\theta_0\cos\theta_1}.$$
Hence by taking derivatives with respect to $y$ and evaluating at $0$ we get
$$t'\left(0\right)= -2\sin\theta_0\cos\theta_1,\quad \theta_0'\left(0\right)= 2\cos\theta_0\cos\theta_1\quad \textrm{and}\quad \theta_1'\left(0\right)= -2\frac{\sin\theta_1}{\sin \theta_0}.$$
Hence $\widetilde{\pi\left(B_2\right)}=-2\sin\theta_0\cos\theta_1\frac{\partial}{\partial t}+2\cos\theta_0\cos\theta_1\frac{\partial}{\partial \theta_0}-2\frac{\sin\theta_1}{\sin\theta_0}\frac{\partial}{\partial \theta_1}$.
\end{proof}

In view of Lemma \ref{ecompu} we get
$$\widetilde{R^{+}}= \left(\cos\theta_0-\sqrt{-1}\sin\theta_0\cos\theta_1\right)\frac{\partial}{\partial t}+ \left(\sin\theta_0+\sqrt{-1}\cos\theta_0\cos\theta_1\right)\frac{\partial}{\partial \theta_0}-\sqrt{-1}\frac{\sin\theta_1}{\sin\theta_0}\frac{\partial}{\partial \theta_1}.$$
Since $R^{+}\varphi_{s,m}=\widetilde{R^{+}}\varphi_{s,m}$, applying $\widetilde{R^{+}}$ to $\varphi_{s,m}$ we get
\begin{equation}
\label{raising}R^{+}\varphi_{s,m}=\left(s+m\right)\varphi_{s,m+1}.
\end{equation}

\begin{Rem}
We note that using the explicit root-space decomposition described as above, one can directly check that the raising operator $R^{+}$ (more explicitly, the matrix representing $R^+$) satisfies
\begin{equation}
\label{1} [H, R^{+}]= \varepsilon_0\left(H\right)R^{+}\quad \textrm{and}\quad [R^{+}, \fk_{\alpha}]= 0
\end{equation}
for any $H\in \mathfrak{h}$ and any $\alpha\in \Phi^{+}$. The first part of $\left(\ref{1}\right)$ implies that $R^{+}$ sends a vector of $K$-weight $\rho$ to a vector of $K$-weight $\rho+ \varepsilon_0$ and the second part of $\left(\ref{1}\right)$ implies that $R^{+}$ sends a highest weight vector to either zero or another highest weight vector. Since $\varphi_{s,m}$ is a highest weight vector of $K$-weight $m\varepsilon_0$, $R^{+}\varphi_{s,m}$ is either zero or a highest weight vector of $K$-weight $\left(m+1\right)\varepsilon_0$. But since $I_{\infty}^s=\oplus_{m=0}^{\infty} I_{\infty}^s\left(K,m\right)$ and each $I_{\infty}^s\left(K,m\right)$ has a unique (up to scalars) highest weight vector $\varphi_{s,m}$, the set of highest weight vectors in $I_{\infty}^s$ is exactly $\{\varphi_{s,m}\ |\ m\geq 0\}$. Thus $R^{+}\varphi_{s,m}$ is a multiple of $\varphi_{s,m+1}$. In fact, $\left(\ref{1}\right)$ is the characterization we used to find $R^{+}$. However, once we have found $R^{+}$, $\left(\ref{1}\right)$ is no longer essential for our proof, since we get $\left(\ref{raising}\right)$ (which trivially implies that $R^{+}\varphi_{s,m}$ is a multiple of $\varphi_{s,m+1}$) by explicit computation.
\end{Rem}

\subsection{Non-spherical Eisenstein series.}
Given $\phi\in L^2\left(M\backslash K,m\right)$, we view $\phi$ as a function on $G$ by setting $\phi\left(uak\right)=\phi\left(k\right)$ for any $u\in N$ and any $a\in A$. We define the non-spherical Eisenstein series by
$$E\left(\phi, s, g\right)= \sum_{\gamma\in \Gamma_{\infty}\backslash \Gamma}\varphi_s\left(\gamma g\right)\phi\left(\gamma g\right).$$
Note that $E\left(\phi, s, g\right)$ is no longer right $K$-invariant. Its constant term (corresponding to the cusp at $\infty$) is defined by
$$E^0\left(\phi, s, g\right)= \frac{1}{|\mathcal{F}_{\mathcal{O}_{\Gamma}}|}\int_{\textbf{x}\in \mathcal{F}_{\mathcal{O}_{\Gamma}}}E\left(\phi, s, u_{\textbf{x}}g\right)d\textbf{x}.$$
Now using $\left(\ref{raising}\right)$ we can get an explicit formula for $E^0\left(\phi, s, g\right)$.
\begin{Prop}
\label{const term} For any $\phi\in L^2\left(M\backslash K, m\right)$,
\begin{equation}
\label{nsconst}E^0\left(\phi, s, g\right)= \left(\varphi_s\left(g\right)+ P_m\left(s\right)\mathcal{C}\left(s\right)\varphi_{n-s}\left(g\right)\right)\phi\left(g\right),
\end{equation}
where $P_0\left(s\right)=1$ and $P_m\left(s\right)= \prod_{k=0}^{m-1}\frac{n-s+k}{s+k}$ for $m\geq 1$.
\end{Prop}

\begin{proof}
For any $m\geq 0$, let $h_m$ be the highest weight vector in $L^2\left(M\backslash K,m\right)$. We first prove $\left(\ref{nsconst}\right)$ for $h_m$. We prove by induction. If $m=0$, then $\left(\ref{nsconst}\right)$ follows from the constant term formula for spherical Eisenstein series. Assume that it holds for some integer $m\geq 0$, we want to show it also holds for $m+1$. We apply the raising operator $R^{+}$ to the constant term $E^0\left(h_m, s, g\right)$. On the one hand, by induction,
$$E^0\left(h_m, s, g\right)= \varphi_{s,m}\left(g\right)+ P_m\left(s\right)\mathcal{C}\left(s\right)\varphi_{n-s,m}\left(g\right).$$
Hence by $\left(\ref{raising}\right)$ we get
$$R^{+}E^0\left(h_m, s, g\right)= \left(s+ m\right)\varphi_{s,m+1}\left(g\right)+ \left(n-s+ m\right)P_m\left(s\right)\mathcal{C}\left(s\right)\varphi_{n-s, m+1}\left(g\right).$$
On the other hand, since $R^{+}$ commutes with the left regular action, we have
\begin{align*}
R^{+}E^0\left(h_m, s, g\right)&= \frac{1}{|\mathcal{F}_{\mathcal{O}_{\Gamma}}|}\int_{\textbf{x}\in \mathcal{F}_{\mathcal{O}_{\Gamma}}}\sum_{\gamma\in \Gamma_{\infty}\backslash \Gamma}R^+\varphi_{s,m}\left(\gamma u_{\textbf{x}}g\right)d\textbf{x}\\
&= \frac{s+m}{|\mathcal{F}_{\mathcal{O}_{\Gamma}}|}\int_{\textbf{x}\in \mathcal{F}_{\mathcal{O}_{\Gamma}}}\sum_{\gamma\in \Gamma_{\infty}\backslash \Gamma}\varphi_{s,m+1}\left(\gamma u_{\textbf{x}}g\right)d\textbf{x}\\
&= \left(s+m\right)E^0\left(h_{m+1},s,g\right).
\end{align*}
Hence
\begin{align*}
E^0\left(h_{m+1},s,g\right)&= \frac{1}{s+m}\left(\left(s+ m\right)\varphi_{s,m+1}\left(g\right)+ \left(n-s+ m\right)P_m\left(s\right)\mathcal{C}\left(s\right)\varphi_{n-s, m+1}\left(g\right)\right)\\
                   &= \varphi_{s,m+1}\left(g\right)+ \frac{n-s+m}{s+ m}P_m\left(s\right)\mathcal{C}\left(s\right)\varphi_{n-s, m+1}\left(g\right)\\
                   &= \varphi_{s,m+1}\left(g\right)+ P_{m+1}\left(s\right)\mathcal{C}\left(s\right)\varphi_{n-s, m+1}\left(g\right).
\end{align*}
Now for general $\phi\in L^2\left(M\backslash K,m\right)$. Since $L^2\left(M\backslash K,m\right)$ is an irreducible $\fk_{\bbC}$-module, $\phi$ can be written as $\phi=\mathcal{D}h_m$ with $\mathcal{D}$ some differential operator on $L^2\left(M\backslash K,m\right)$ generated by $\pi\left(\fk_{\bbC}\right)$. Since $\pi\left(\fk_{\bbC}\right)$ acts trivially on the character $\varphi_s$, we have $\mathcal{D}\varphi_{s,m}=\varphi_s\mathcal{D}h_m= \varphi_s\phi$. Hence on the one hand,
\begin{align*}
\mathcal{D}E^0\left(h_m, s, g\right)&= \frac{1}{|\mathcal{F}_{\mathcal{O}_{\Gamma}}|}\int_{\textbf{x}\in \mathcal{F}_{\mathcal{O}_{\Gamma}}}\sum_{\gamma\in \Gamma_{\infty}\backslash \Gamma}\mathcal{D}\varphi_{s,m}\left(\gamma u_{\textbf{x}}g\right)d\textbf{x}\\
&= \frac{1}{|\mathcal{F}_{\mathcal{O}_{\Gamma}}|}\int_{\textbf{x}\in \mathcal{F}_{\mathcal{O}_{\Gamma}}}\sum_{\gamma\in \Gamma_{\infty}\backslash \Gamma}\varphi_s\left(\gamma u_{\textbf{x}}g\right)\phi\left(\gamma u_{\textbf{x}}g\right)d\textbf{x}\\
&= E^0\left(\phi,s,g\right).
\end{align*}
On the other hand, using $\left(\ref{nsconst}\right)$ for $h_m$ we get
\begin{align*}
\mathcal{D}E^0\left(h_m,s,g\right)&=\mathcal{D}\left(\left(\varphi_s\left(g\right)+ P_m\left(s\right)\mathcal{C}\left(s\right)\varphi_{n-s}\left(g\right)\right)h_m\left(g\right)\right)\\
&=\left(\varphi_s\left(g\right)+ P_m\left(s\right)\mathcal{C}\left(s\right)\varphi_{n-s}\left(g\right)\right)\phi\left(g\right).
\end{align*}
This completes the proof.
\end{proof}

\section{Bounds for incomplete Eisenstein series}
\subsection{Explicit formula.}
For each $L^2\left(M\backslash K, m\right)$ we fix an orthonormal basis
$$\left\{\psi_{m,l}\ |\ 1\leq l\leq \textrm{dim}\ L^2\left(M\backslash K, m\right)\right\}.$$
For any $f\in C_{c}^{\infty}\left(Q\backslash G\right)$ let
$$\hat{f}_{m,l}\left(a\right)= \int_K f\left(ak\right)\overline{\psi_{m,l}\left(k\right)}dk,$$ and we define the following function
$$M_f\left(s\right)= \sum_{m,l}P_m\left(s\right)\big|\int_{\bbR}\hat{f}_{m,l}\left(a_t\right)e^{-st}dt\big|^2,$$
with $P_m\left(s\right)$ as in Proposition \ref{const term}. We then have
\begin{Prop}$($cf. \cite[Proposition 2.3]{DA}$)$
\label{p.bound} Let $\frac{n}{2}< s_p< \cdots< s_1< s_0= n$ denote the poles of $\mathcal{C}\left(s\right)$ and let $c_j= \textrm{Res}_{s=s_j}\mathcal{C}\left(s\right)$ be the residue of $\mathcal{C}\left(s\right)$ at $s_j$ for $0\leq j\leq p$. Then for any $f\in C^{\infty}_{c}\left(Q\backslash G\right)$ we have
$$||\Theta_f||_2^2\leq \frac{|\mathcal{F}_{\mathcal{O}_{\Gamma}}|}{[\Gamma_{\infty}: \Gamma_{\infty}']\nu_{\Gamma}}\left(2||f||_2^2+ c_0||f||_1^2+ \sum_{j=1}^pc_jM_f\left(s_j\right)\right).$$
\end{Prop}

Note that $f\in C^{\infty}_{c}\left(Q\backslash G\right)$ can be written as $f= \sum_{m,l}f_{m,l}$ with $f_{m,l}\left(ak\right)= \hat{f}_{m,l}\left(a\right)\psi_{m,l}\left(k\right)$. We first prove a preliminary estimate for each $f_{m,l}$ and then deduce the proposition from this estimate.
\begin{Lem}
\label{l.bound} Let $f\in C_{c}^{\infty}\left(Q\backslash G\right)$ be of the form $f\left(a_tk\right)= v\left(t\right)\phi\left(k\right)$ where $v\left(t\right)\in C_c^{\infty}\left(\bbR\right)$ and $\phi\in L^2\left(M\backslash K, m\right)$ for some $m$. Let $s_j$ and $c_j$ be as above, we then have
$$||\Theta_f||_2^2\leq \frac{|\mathcal{F}_{\mathcal{O}_{\Gamma}}|}{[\Gamma_{\infty}: \Gamma_{\infty}']\nu_{\Gamma}}\left(2||f||_2^2+ \sum_{j=0}^{p}c_jP_m\left(s_j\right)||\phi||_2^2\left|\int_{\bbR}v\left(t\right)e^{-s_jt}dt\right|^2\right).$$
\end{Lem}
\begin{proof}
Let $\hat{v}\left(r\right)= \frac{1}{\sqrt{2\pi}}\int_{\bbR}v\left(t\right)e^{-irt}dt$ denote the Fourier transform of $v$. Note that $\hat{v}\left(r\right)$ extends to an entire function in $r$ since $v$ is smooth and compactly supported. Recall the Fourier inversion formula $$v\left(t\right)= \frac{1}{\sqrt{2\pi}}\int_{\bbR}\hat{v}\left(r\right)e^{itr}dr.$$ Making the substitution $s=ir$ we get
$$v\left(t\right)= \frac{1}{\sqrt{2\pi}i}\int_{0-i\infty}^{0+i\infty}\hat{v}\left(-is\right)e^{st}ds.$$ For any $\sigma > n$ shifting the contour of integration to the line $\textrm{Re}\left(s\right)= \sigma$ we get
$$v\left(t\right)= \frac{1}{\sqrt{2\pi}i}\int_{\sigma-i\infty}^{\sigma+i\infty}\hat{v}\left(-is\right)e^{st}ds= \frac{1}{\sqrt{2\pi}i}\int_{\sigma-i\infty}^{\sigma+i\infty}\hat{v}\left(-is\right)\varphi_{s}\left(a_t\right)ds.$$ Consequently we can write
$$f\left(g\right)= \frac{1}{\sqrt{2\pi}i}\int_{\sigma-i\infty}^{\sigma+i\infty}\hat{v}\left(-is\right)\varphi_{s}\left(g\right)\phi\left(g\right)ds$$ and summing over $\Gamma_{\infty}\backslash \Gamma$ we get
\begin{align*}
\Theta_f\left(g\right)&= \frac{1}{\sqrt{2\pi}i}\int_{\sigma-i\infty}^{\sigma+i\infty}\hat{v}\left(-is\right)\sum_{\gamma\in \Gamma_{\infty}\backslash \Gamma}\varphi_{s}\left(\gamma g\right)\phi\left(\gamma g\right)ds\\
   &= \frac{1}{\sqrt{2\pi}i}\int_{\sigma-i\infty}^{\sigma+i\infty}\hat{v}\left(-is\right)E\left(\phi, s, g\right)ds.
\end{align*}
Integrating this over $\mathcal{F}_{\mathcal{O}_{\Gamma}}$ gives
\begin{align*}
\int_{\mathcal{F}_{\mathcal{O}_{\Gamma}}}\Theta_f\left(u_{\textbf{x}}ak \right)d\textbf{x}&= \frac{1}{\sqrt{2\pi}i}\int_{\sigma-i\infty}^{\sigma+i\infty}\hat{v}\left(-is\right)\int_{\mathcal{F}_{\mathcal{O}_{\Gamma}}}E\left(\phi, s, u_{\textbf{x}}ak\right)d\textbf{x}ds\\
   &= \frac{|\mathcal{F}_{\mathcal{O}_{\Gamma}}|}{\sqrt{2\pi}i}\int_{\sigma-i\infty}^{\sigma+i\infty}\hat{v}\left(-is\right)E^{0}\left(\phi, s, ak\right)ds.
\end{align*}
Using the formula $\left(\ref{nsconst}\right)$ for $E^{0}\left(\phi, s, g\right)$ we get
\begin{align*}
\int_{\mathcal{F}_{\mathcal{O}_{\Gamma}}}\Theta_f\left(u_{\textbf{x}}ak \right)d\textbf{x}=&\frac{|\mathcal{F}_{\mathcal{O}_{\Gamma}}|\phi\left(k\right)}{\sqrt{2\pi}i}
\bigg(\int_{\sigma-i\infty}^{\sigma+i\infty}\hat{v}\left(-is\right)\varphi_s\left(a\right)ds\\
&+\int_{\sigma-i\infty}^{\sigma+i\infty}\hat{v}(-is)\mathcal{C}(s)P_m\left(s\right)\varphi_{n-s}\left(a\right)ds\bigg).
\end{align*}
Now shift the contour of integration to the line $\textrm{Re}\left(s\right)= \frac{n}{2}$ (picking up possible poles) to get
\begin{align*}
\int_{\mathcal{F}_{\mathcal{O}_{\Gamma}}}\Theta_f\left(u_{\textbf{x}}ak \right)d\textbf{x}=&\frac{|\mathcal{F}_{\mathcal{O}_{\Gamma}}|\phi\left(k\right)}{\sqrt{2\pi}i}
\bigg(\int_{\frac{n}{2}-i\infty}^{\frac{n}{2}+i\infty}\hat{v}\left(-is\right)\varphi_s\left(a\right)ds\\
& +\int_{\frac{n}{2}-i\infty}^{\frac{n}{2}+i\infty}\hat{v}\left(-is\right)\mathcal{C}\left(s\right)P_m\left(s\right)\varphi_{n-s}\left(a\right)ds\\
&+ 2\pi i\sum_{j=0}^{p}c_jP_m\left(s_j\right)\hat{v}\left(-is_j\right)\varphi_{n-s_j}\left(a\right)\bigg),
\end{align*}
where $c_j= \textrm{Res}_{s=s_j}\mathcal{C}\left(s\right)$ is the residue of $\mathcal{C}\left(s\right)$ at the pole $s_j$.
Now applying formula $\left(\ref{L2}\right)$ we get
\begin{align*}
||\Theta_f||_2^2=& \frac{1}{[\Gamma_{\infty}: \Gamma_{\infty}']\nu_{\Gamma}}\int_{K}\int_{\bbR}\overline{f\left(a_tk\right)}e^{-nt}\int_{\mathcal{F}_{\mathcal{O}_{\Gamma}}}\Theta_f\left(u_{\textbf{x}}a_tk \right)d\textbf{x}dtdk\\
                =& \frac{|\mathcal{F}_{\mathcal{O}_{\Gamma}}|}{[\Gamma_{\infty}: \Gamma_{\infty}']\nu_{\Gamma}}\int_K|\phi\left(k\right)|^2dk\frac{1}{\sqrt{2\pi}i}\bigg(\int_{\frac{n}{2}-i\infty}^{\frac{n}{2}+i\infty}\hat{v}\left(-is\right)\int_{\bbR}\overline{v\left(t\right)}e^{-nt}\varphi_s\left(a_t\right)dtds\\
                &+ \int_{\frac{n}{2}-i\infty}^{\frac{n}{2}+i\infty}\hat{v}\left(-is\right)\mathcal{C}\left(s\right)P_m\left(s\right)\int_{\bbR}\overline{v\left(t\right)}e^{-nt}\varphi_{n-s}\left(a_t\right)dtds\\
                & + 2\pi i\sum_{j=0}^{p}c_jP_m\left(s_j\right)\hat{v}\left(-is_j\right)\int_{\bbR}\overline{v\left(t\right)}e^{-nt}\varphi_{n-s_j}\left(a_t\right)dt\bigg).
\end{align*}
Note that for $s\in \bbC$ we have
$$\frac{1}{\sqrt{2\pi}}\int_{\bbR}\overline{v\left(t\right)}e^{-nt}\varphi_s\left(a_t\right)dt= \overline{\hat{v}\left(i\left(\bar{s}-n\right)\right)}.$$
Hence (use the substitution $s= \frac{n}{2}+ ir$)
\begin{align*}
 ||\Theta_f||_2^2=& \bigg(\int_{\bbR}|\hat{v}(r-\frac{n}{2}i)|^2dr+ \int_{\bbR}\hat{v}(r-\frac{n}{2}i)\overline{\hat{v}(-r-\frac{n}{2}i)}\mathcal{C}(\frac{n}{2}+ ir)P_m(\frac{n}{2}+ ir)dr\\
 &
 +2\pi \sum_{j=0}^{p}c_jP_m\left(s_j\right)|\hat{v}\left(-is_j\right)|^2\bigg)\frac{|\mathcal{F}_{\mathcal{O}_{\Gamma}}|}{[\Gamma_{\infty}: \Gamma_{\infty}']\nu_{\Gamma}}||\phi||_2^2.
 \end{align*}
 Now for the first term, applying Plancherel's theorem for $v\left(t\right)e^{-\frac{n}{2}t}$ we get
 $$||\phi||_2^2\int_{\bbR}|\hat{v}\left(r-\frac{n}{2}i\right)|^2dr= ||\phi||_2^2\int_{\bbR}|v\left(t\right)|^2e^{-nt}dt= ||f||_2^2.$$ For the second term, using the fact that $|\mathcal{C}\left(\frac{n}{2}+ir\right)|\leq 1$, $|P_m\left(\frac{n}{2}+ ir\right)|=1$ and Cauchy-Schwarz, we see that its absolute value is bounded by the first term. Finally, for the last term we have for each pole $2\pi |\hat{v}\left(-is_j\right)|^2= |\int_{\bbR}v\left(t\right)e^{-s_jt}dt|^2$. This finishes the proof.
\end{proof}

We can now give the proof of Proposition $\ref{p.bound}$. Note that for $f= \sum_{m,l}f_{m,l}$ as above, by orthogonality we have
$$||\Theta_f||_2^2= \sum_{m,l}||\Theta_{f_{m,l}}||_2^2.$$
We can use Lemma $\ref{l.bound}$ to estimate each of the terms $||\Theta_{f_{m,l}}||_2^2$ separately and sum all the contributions.

First, for the $L^2$-norms we have $\sum_{m,l}||f_{m,l}||_2^2= ||f||_2^2$. Next, the contribution of the exceptional poles $\frac{n}{2}<s_j< n$ is $\sum_{j=1}^{p} c_jM_f\left(s_j\right)$. Finally, for the pole $s_0=n$, note that $P_m\left(n\right)= 0$ except $m=0$. Thus its contribution is
$$c_0\bigg|\int_{\bbR}\hat{f}_{0,0}\left(a_t\right)e^{-nt}dt\bigg|^2= c_0\bigg|\int_{Q\backslash G}f\left(g\right)dg\bigg|^2\leq c_0||f||_1^2.$$

\begin{Rem}
We note that if $f$ is of the form $f\left(a_tk\right)=v\left(t\right)\phi\left(k\right)$, with $v$ smooth, compactly supported and $\phi$ any function in $L^2\left(M\backslash K\right)$, then Proposition \ref{p.bound} still holds by exactly the same computation.
\end{Rem}

\subsection{Proof of Theorem \ref{thm 2}}\label{proof1.2}
Fix a parameter $\lambda>0$, define $\mathscr{A}_{\lambda}\subset L^2\left(Q\backslash G\right)$ to be the set of functions of the form
$$f\left(a_tk\right)= v\left(t\right)\phi\left(k\right)$$
with $v$ smooth, nonnegative, compactly supported and satisfying \begin{equation}
\label{ff} \frac{\int_{\bbR} v\left(t\right)e^{-st}dt}{\left(\int_{\bbR} v\left(t\right)e^{-nt}dt\right)^{\frac{2s}{n}-1}\left(\int_{\bbR} v^2\left(t\right)e^{-nt}dt\right)^{1-\frac{s}{n}}}\leq \lambda
\end{equation}
for any $s\in\left(\frac{n}{2},n\right)$, and $\phi$ any function in $L^2\left(M\backslash K\right)$. In this section we will show that for any $s\in\left(\frac{n}{2},n\right)$,
\begin{equation}
\label{bbound}M_f\left(s\right)\lesssim_{s,\lambda}||f||_1^2+||f||_2^2
\end{equation}
for all $f\in\mathscr{A}_{\lambda}$. In particular, this bound holds at the finitely many exceptional poles $s_j$ determined by $\Gamma$. Hence in view of Proposition \ref{p.bound} it implies Theorem \ref{thm 2}.

To show $\left(\ref{bbound}\right)$, we first give two preliminary lemmas.
\begin{Lem}
\label{new form} If $f\in L^2\left(Q\backslash G\right)$ is of the form $f\left(a_tk\right)= v\left(t\right)\phi\left(k\right)$, then $M_f\left(s\right)$ can be written as
$$M_f\left(s\right)=\left(\sum_{m=0}^{\infty}P_m\left(s\right)||\phi_m||_2^2\right)\left|\int_{\bbR}v\left(t\right)e^{-st}dt\right|^2,$$
where $\phi_m$ denotes the projection of $\phi$ into $L^2\left(M\backslash K, m\right)$.
\end{Lem}
\begin{proof}
For each $m\geq 0$, let
$$\left\{\psi_{m,l}\ |\ 1\leq l\leq \textrm{dim}L^2\left(M\backslash K, m\right)\right\}$$
be the fixed orthonormal basis of $L^2\left(M\backslash K, m\right)$ as before. Then we have
$$\phi_m\left(k\right)= \sum_{l}c_{m,l}\psi_{m,l}\left(k\right)$$
with $c_{m,l}= \int_K \phi\left(k\right)\overline{\psi_{m,l}\left(k\right)}dk$. Moreover, we have
$$||\phi_m||_2^2= \sum_{l}\left|c_{m,l}\right|^2$$
and
$$\hat{f}_{m,l}\left(a_t\right)= v\left(t\right)\int_K \phi\left(k\right)\overline{\phi_{m,l}\left(k\right)}dk= c_{m,l}v\left(t\right).$$
Hence
\begin{align*}
M_f\left(s\right)&=\sum_{m,l}P_m\left(s\right)\left|\int_{\bbR}\hat{f}_{m,l}\left(a_t\right)e^{-st}dt\right|^2\\
      &=\sum_{m=0}^{\infty}P_m\left(s\right)\sum_{l}|c_{m,l}|^2\left|\int_{\bbR}v\left(t\right)e^{-st}dt\right|^2\\
      &= \left(\sum_{m=0}^{\infty}P_m\left(s\right)||\phi_m||_2^2\right)\left|\int_{\bbR}v\left(t\right)e^{-st}dt\right|^2.\qedhere
\end{align*}
\end{proof}

\begin{Lem}
For any $s\in \left(\frac{n}{2}, n\right)$, $P_m\left(s\right)\asymp_{s}\left(m+1\right)^{\left(n-2s\right)}$.
\end{Lem}

\begin{proof}
Since $P_m\left(s\right)= \prod_{k=0}^{m-1}\frac{n-s+k}{s+ k}$ we have
\begin{align*}
\log\left(P_m\left(s\right)\right)&= \log\left(\frac{n-s}{s}\right)+ \sum_{k=1}^{m-1}\left(\log\left(1+ \frac{n-s}{k}\right)-\log\left(1+\frac{s}{k}\right)\right)\\
            &=\left(n-2s\right)\sum_{k=1}^{m-1}\frac{1}{k}+ O_s\left(1\right)=\left(n-2s\right)\log\left(m+1\right)+ O_s\left(1\right). \qedhere
\end{align*}
\end{proof}
Thus for $f\left(a_tk\right)=v\left(t\right)\phi\left(k\right)\in \mathscr{A}_{\lambda}$ we define
$$\tilde{M}_{f}\left(s\right):= \left(\int_{\bbR} v\left(t\right)e^{-st}dt\right)^2\sum_{m=0}^{\infty}\frac{||\phi_m||_2^2}{\left(m+1\right)^{\left(2s-n\right)}},$$
where $\phi_m$ is the projection of $\phi$ into $L^2\left(M\backslash K,m\right)$.
In view these two lemmas it suffices to prove $\tilde{M}_{f}\left(s\right)\lesssim_{s,\lambda}||f||_1^2+||f||_2^2$ for all $f\in \mathscr{A}_{\lambda}$.

\begin{proof}[Proof of Theorem \ref{thm 2}]
We first prove for $f=v\phi\in \mathscr{A}_{\lambda}$ with $||\phi||_2\leq ||\phi||_1$. We first recall a simple inequality that for any $y_1, y_2>0$ and $0<\eta<1$, $y_1^{\eta}y_2^{1-\eta}\leq \max\left(y_1,y_2\right)\leq y_1+y_2$. Hence in view of $\left(\ref{ff}\right)$, for any $s\in \left(\frac{n}{2},n\right)$, since $\left(\frac{2s}{n}-1\right)+ \left(2-\frac{2s}{n}\right)=1$, we have
$$\left(\int_{\bbR}v\left(t\right)e^{-st}dt\right)^2\leq \lambda^2\left(\left(\int_{\bbR}v\left(t\right)e^{-nt}dt\right)^2+\int_{\bbR}v^2\left(t\right)e^{-nt}dt\right).$$
Thus
\begin{align*}
\tilde{M}_{f}\left(s\right)&= \left(\int_{\bbR} v\left(t\right)e^{-st}dt\right)^2\sum_{m=0}^{\infty}\frac{||\phi_m||_2^2}{\left(m+1\right)^{\left(2s-n\right)}}\\
&\leq\lambda^2\left(\left(\int_{\bbR}v\left(t\right)e^{-nt}dt\right)^2+\int_{\bbR}v^2\left(t\right)e^{-nt}dt\right)||\phi||_2^2\\
&\leq\lambda^2 \left(\left(\int_{\bbR}v\left(t\right)e^{-nt}dt\right)^2||\phi||_1^2+ \int_{\bbR}v^2\left(t\right)e^{-nt}dt||\phi||_2^2\right)\\
&=\lambda^2\left(||f||_1^2+||f||_2^2\right).
\end{align*}
Now we prove the case $||\phi||_2>||\phi||_1$. Let $\iota:=\frac{||\phi||_2}{||\phi||_1}>1$. We separate the summation into two parts:
$$\sum_{m=0}^{\infty}\frac{||\phi_m||_2^2}{\left(m+1\right)^{\left(2s-n\right)}}=\left(\sum_{m=0}^{\floor{\iota^{\frac{2}{n}}}}+\sum_{m=\floor{\iota^{\frac{2}{n}}}+1}^{\infty}\right)
\frac{||\phi_m||_2^2}{\left(m+1\right)^{\left(2s-n\right)}}.$$
For the first part, we invoke an estimate from spherical harmonic analysis (\cite[inequality (4.4)]{Ch}).\footnotemark \ Namely, for any $\phi\in L^2\left(M\backslash K\right)$ and $m\geq 0$,
$$||\phi_m||^2_2\lesssim \left(m+1\right)^{n-1}||\phi||^2_1$$
with the implicit constant only depends on the dimension of $M\backslash K$. Thus we have
\footnotetext{The exact form of inequality $\left(4.4\right)$ in \cite{Ch} is $||\phi_m||_2\lesssim m^{\frac{n-1}{2}}||\phi||_1$. Here we square both sides and replace $m$ by $m+1$ to cover the case $m=0$.}
\begin{align*}
\sum_{m=0}^{\floor{\iota^{\frac{2}{n}}}}\frac{||\phi_m||_2^2}{\left(m+1\right)^{\left(2s-n\right)}}&\lesssim \sum_{m=0}^{\floor{\iota^{\frac{2}{n}}}}\frac{||\phi||_1^2}{\left(m+1\right)^{1-2\left(n-s\right)}}\\
&\asymp_{s} \left(\iota^{\frac{2}{n}}\right)^{2\left(n-s\right)}||\phi||_1^2\quad\left(\textrm{since $1-2\left(n-s\right)<1$}\right)\\
&=||\phi||_1^{2\left(\frac{2s}{n}-1\right)}||\phi||_2^{4\left(1-\frac{s}{n}\right)}.
\end{align*}
For the second part, we have
\begin{align*}
\sum_{m=\floor{\iota^{\frac{2}{n}}}+1}^{\infty}\frac{||\phi_m||_2^2}{\left(m+1\right)^{\left(2s-n\right)}}&\leq \frac{1}{\left(\iota^{\frac{2}{n}}\right)^{\left(2s-n\right)}}\sum_{m=\floor{\iota^{\frac{2}{n}}}+1}^{\infty}||\phi_m||_2^2\\
&\leq \iota^{-2\left(\frac{2s}{n}-1\right)}||\phi||_2^2\\
&=||\phi||_1^{2\left(\frac{2s}{n}-1\right)}||\phi||_2^{4\left(1-\frac{s}{n}\right)}.
\end{align*}
Hence
\begin{equation}
\label{fsbound}\sum_{m=0}^{\infty}\frac{||\phi_m||_2^2}{\left(m+1\right)^{\left(2s-n\right)}}\lesssim_{s}||\phi||_1^{2\left(\frac{2s}{n}-1\right)}||\phi||_2^{4\left(1-\frac{s}{n}\right)}.
\end{equation}
Thus by $\left(\ref{ff}\right)$ and $\left(\ref{fsbound}\right)$ we get
\begin{align*}
\tilde{M}_{f}\left(s\right)&= \left(\int_{\bbR} v\left(t\right)e^{-st}dt\right)^2\sum_{m=0}^{\infty}\frac{||\phi_m||_2^2}{\left(m+1\right)^{\left(2s-n\right)}}\\
&\lesssim_{s,\lambda} \left(\int_{\bbR} v\left(t\right)e^{-nt}dt\right)^{2\left(\frac{2s}{n}-1\right)}\left(\int_{\bbR} v^2\left(t\right)e^{-nt}dt\right)^{2-\frac{2s}{n}}||\phi||_1^{2\left(\frac{2s}{n}-1\right)}||\phi||_2^{4\left(1-\frac{s}{n}\right)}\\
&=||f||_1^{2\left(\frac{2s}{n}-1\right)}||f||_2^{2\left(2-\frac{2s}{n}\right)}\\
&\leq ||f||_1^2+||f||_2^2.
\end{align*}
This finishes the proof.
\end{proof}

\section{Logarithm Laws}\label{loglaw}
Fix $o\in \Gamma\backslash G$ and let $\{g_t\}\subset G$ be a one-parameter unipotent subgroup in $G$. Let $\textrm{dist}_G$ and $\textrm{dist}_{\Gamma}$ be the hyperbolic distance functions on $G/K$ and $\Gamma\backslash G/K$ respectively. In this section we will prove logarithm laws for the unipotent flow $\{g_t\}$, that is
\begin{equation}
\label{ll} \limsup\limits_{t\to \infty}\frac{\textrm{dist}_{\Gamma}\left(o, xg_t\right)}{\log t}= \frac{1}{n}
\end{equation}
for $\sigma$-a.e. $x\in \Gamma\backslash G$.
First we note that if $\left(\ref{ll}\right)$ holds for $\Gamma$, then it also holds for any $\Gamma'= g^{-1}\Gamma g$. This follows from the following identity
$$\textrm{dist}_{\Gamma}\left(\Gamma h, \Gamma h'\right)= \textrm{dist}_{\Gamma'}\left(\Gamma'g^{-1}h, \Gamma'g^{-1}h'\right),$$
where $h,h'$ are any two elements in $G$ and $\textrm{dist}_{\Gamma'}$ is the hyperbolic distance function on $\Gamma'\backslash G/K$. Hence we can assume that $\Gamma$ has a cusp at $\infty$. Fix this $\Gamma$ and we denote $\textrm{dist}=\textrm{dist}_{\Gamma}$ without ambiguity. Next note that $\{g_t\}$ can be replaced by a new flow $\{\tilde{g}_{t}\}$ with $\tilde{g}_{t}= k^{-1}g_{\eta t}k$ for some $k\in K$ and $\eta> 0$. This is because
$$\limsup\limits_{t\to \infty}\frac{\textrm{dist}\left(o, x\tilde{g}_{t}\right)}{\log t}=\limsup\limits_{t\to \infty}\frac{\textrm{dist}\left(o, x'g_{t}\right)}{\log t}$$
with $x'= xk^{-1}$. For any $\textbf{x}\in\mathbb{V}^{n-1}$, denote $u_{\textbf{x}}^-=\begin{pmatrix}
1 & 0\\
\textbf{x} & 1\end{pmatrix}$ to be the corresponding lower triangular unipotent matrix.
\begin{Lem}
\label{slem}Every unipotent element in $G$ is $K$-conjugate to $u_{x}^-$ for some $x>0$.
\end{Lem}
\begin{proof}
Let $g$ be an unipotent element in $G$. We first note that it suffices to show $g$ is $K$-conjugate to $u_{-x}$ for some $x>0$. This is because $u_x^-$ is conjugate to $u_{-x}$ via $\begin{pmatrix}
0 & -1\\
1 & 0\end{pmatrix}$ and $\begin{pmatrix}
0 & -1\\
1 & 0\end{pmatrix}\in K$. Next we note that since $g$ is unipotent, $g$ is conjugate to some element in $N$. By Iwasawa decomposition and the fact that $NA$ normalizes $N$, $g$ is $K$-conjugate to some element in $N$. Hence we can assume $g$ is contained in $N$. Finally, we note that any element in $N$ is conjugate to some $u_{-x}$ with $x>0$ via the group $M$ since the conjugation action of $M$ on $N$ realizes $M$ as the rotation group of $N$.
\end{proof}
Since we can conjugate $\{g_t\}$ by some element in $K$ and rescale it by a positive number, in view of Lemma \ref{slem} we can assume the unipotent flow is given by $\left\{g_t=u_t^-\right\}_{t\in\bbR}.$

\subsection{Technical lemmas.}\label{tl}
For any $\mathfrak{D}\subset Q\backslash G$, we denote $|\mathfrak{D}|$ to be its measure with respect to the right $G$-invariant measure on $Q\backslash G$ as fixed in $\left(\ref{Q}\right)$. We define the set $Y_{\mathfrak{D}}\subset \Gamma\backslash G$ corresponding to $\mathfrak{D}$ by
$$ Y_{\mathfrak{D}}:=\left\{ \Gamma g\in \Gamma\backslash G\ \big|\ Q\gamma g\in \mathfrak{D}\ \text{for some}\ \gamma\in \Gamma\right\}.$$

Let $\{r_{\ell}\}$ be any sequence of positive numbers such that $r_{\ell}\to \infty$ and $\sum_{\ell=1}^{\infty}e^{-nr_{\ell}}=\infty$. For any integer $m\geq 1$, let $p\left(m\right)>m$ be an integer such that $\sum_{\ell=m}^{p\left(m\right)}e^{-nr_{\ell}}\geq 1$ for all $m\geq 1$.  Let $N^{-}$ be the subgroup of lower triangular unipotent matrices and $$B^{-}=\{u_{\textbf{x}}^{-}\ |\ \textbf{x}\in \mathbb{V}^{n-1}, |\textbf{x}|< \frac12\}$$
be the open ball with radius $\frac12$, centered at the identity in $N^{-}$. We define the set
$$\mathfrak{D}_m:= Q\backslash \bigcup _{\ell=m}^{p\left(m\right)}QA\left(r_{\ell}\right)B^{-}g_{-\ell}\subset Q\backslash G,$$
where $A\left(\tau\right)= \{a_t\ |\ t\geq \tau\}$.

\begin{Lem}
\label{lem 2} $\left\{|\mathfrak{D}_m|\right\}_{m\geq 1}$ is uniformly bounded from below.
\end{Lem}
\begin{proof}
Note that every matrix $\begin{pmatrix}
 a& b\\
 c& d
 \end{pmatrix}\in G$ with $d\neq 0$ can be written uniquely as
 $$\begin{pmatrix}
 a& b\\
 c& d
 \end{pmatrix}= \begin{pmatrix}
 1& bd^{-1}\\
 0& 1
 \end{pmatrix}\begin{pmatrix}
 \frac{d'}{|d|}& 0\\
 0& \frac{d}{|d|}
 \end{pmatrix}\begin{pmatrix}
 |d|^{-1}& 0\\
 0& |d|
 \end{pmatrix}\begin{pmatrix}
 1& 0\\
 d^{-1}c& 1
 \end{pmatrix}.$$ Hence $NMAN^{-}=\left\{\begin{pmatrix}
 a& b\\
 c& d
 \end{pmatrix}\in G\ |\ d\neq 0\right\}$ is a Zariski open dense subset of $G$. Thus there is a Zariski open dense subset in $Q\backslash G$ which can be expressed by the coordinates $Qg= Qa_tu_{\textbf{x}}^{-}$. We note that in these coordinates, the right $G$-invariant measure on $Q\backslash G$ (up to scalars) is given by $e^{-nt}dtd\textbf{x}$ since this is the right Haar measure for the group $AN^{-}$. Moreover, $\mathfrak{D}_m$ is a disjoint union of the sets $Q\backslash QA(r_{\ell})B^-g_{-\ell}$ since $B^-g_{-\ell_1}\cap B^{-}g_{-\ell_2}=\emptyset$ whenever $\ell_1\neq \ell _2$. Hence one can compute
 \begin{align*}
 |\mathfrak{D}_m|&= \sum_{\ell=m}^{p\left(m\right)}\int_{r_{\ell}}^{\infty}e^{-nt}dt\int_{B^{-}g_{-\ell}}1d\textbf{x}\\
                         &\asymp \sum_{\ell=m}^{p\left(m\right)}e^{-nr_{\ell}}\geq 1.\qedhere
 \end{align*}
 \end{proof}

 \begin{Lem}
 \label{lem 1} There is some sufficiently large integer $L$ such that for any $m\geq L$ and any $x\in Y_{\mathfrak{D}_m}$ there exists $m\leq \ell\leq p\left(m\right)$ such that
 $$\textrm{dist}\left(o, xg_{\ell}\right)\geq r_{\ell}+ O\left(1\right).$$
\end{Lem}
\begin{proof}
First recall the Siegel fundamental domain $\mathcal{F}_{\Gamma, \tau_0,U_0}$ we fixed in $\left(\ref{sf}\right)$, take $L$ such that $r_{\ell}-\log 2\geq \tau_0$ for all $\ell\geq L$.
Next using $\left(\ref{extract 1}\right)$ we have that for any $\tau\in \bbR$
 \begin{equation}
 \label{so} QA\left(\tau\right)B^{-}\subset QA\left(\tau- \log 2\right)K= NA\left(\tau-\log 2\right)K.
 \end{equation}
Hence for $m\geq L$, $x\in Y_{\mathfrak{D}_m}$ can be written as $x= \Gamma gg_{-\ell}$ for some $m\leq \ell\leq p\left(m\right)$ with $g\in QA\left(r_{\ell}\right)B^{-}\subset NA\left(r_{\ell}-\log 2\right)K$. After left multiplying by some $\gamma\in \Gamma_{\infty}'$ we can assume that $g=ua_tk$ is contained in the Siegel set $U_0 A\left(r_{\ell}-\log 2\right)K$ (we can do this since $U_0$ contains a fundamental domain of $\Gamma_{\infty}'\backslash N$). Since $\ell\geq m\geq L$, $r_{\ell}-\log 2\geq \tau_0$. By $\left(\ref{final 2}\right)$ we have
\begin{displaymath}
\textrm{dist}\left(o, xg_{\ell}\right)= \textrm{dist}\left(o, ua_tk\right)= t+ O\left(1\right)\geq r_{\ell}-\log 2+O\left(1\right)= r_{\ell}+ O\left(1\right).\qedhere
\end{displaymath}
\end{proof}

The next lemma shows that there exists a nice set sitting inside $\mathfrak{D}_m$. We first identify $Q\backslash G$ with $A\times M\backslash K$. Let $\textrm{Pr}: Q\backslash G\to M\backslash K$ be the natural projection map and
$$K\left(\ell\right):= \textrm{Pr}\left(Q\backslash QA\left(r_{\ell}\right)B^{-}g_{-\ell}\right)$$
be the $K$-part of $Q\backslash QA\left(r_{\ell}\right)B^{-}g_{-\ell}$. We note that $K\left(\ell\right)$ is independent of $r_{\ell}$, that is, $K\left(\ell\right)=\textrm{Pr}\left(Q\backslash QA\left(\tau\right)B^{-}g_{-\ell}\right)$ for any $\tau\in \bbR$.

\begin{Lem}
\label{lem 3} For any $\ell\geq 1$, let $\tau_{\ell}= r_{\ell}-2\log\left(\ell\right)+\log 2$. Then
$$A\left(\tau_{\ell}\right)\times K\left(\ell\right)\subset Q\backslash QA\left(r_{\ell}\right)B^{-}g_{-\ell}\quad and\quad |A\left(\tau_{\ell}\right)\times K\left(\ell\right)|\asymp |Q\backslash QA\left(r_{\ell}\right)B^{-}g_{-\ell}|$$
with the implicit constant independent of $\ell$.
\end{Lem}
\begin{proof}
For each $k\in K\left(\ell\right)$, define
$$I\left(k\right):= \{t\in \bbR\ |\ Qa_{t}k\in Q\backslash QA\left(r_{\ell}\right)B^{-}g_{-\ell}\}$$
and
$$t\left(k\right):= \inf I\left(k\right).$$ We first note that if $t\in I\left(k\right)$ (that is, $a_{t}k= qa_{t_0}u_{\textbf{x}-\ell}^-$ for some $q\in Q$, $t_0\geq r_{\ell}$ and $|\textbf{x}|< \frac12$), then $[t,\infty)\subset I\left(k\right)$. This is because for any $t'> t$ we have
$$a_{t'}k= a_{t'-t}a_{t}k=a_{t'-t}qa_{t_0}u_{\textbf{x}-\ell}^-= q'a_{t'-t+ t_0}u_{\textbf{x}-\ell}^-,$$
and $t'-t+ t_0> t_0\geq r_{\ell}$. Here $q'= a_{t'-t}qa_{t-t'}\in Q$. This implies that
\begin{equation}
\label{estimate set} Q\backslash QA\left(r_{\ell}\right)B^{-}g_{-\ell}=\bigcup_{k\in K\left(\ell\right)}A\left(t\left(k\right)\right)\times \{k\}.
\end{equation}
Moreover, by  $\left(\ref{extract 1}\right)$, the relation $a_{t}k= qa_{t_0}u_{\textbf{x}-\ell}^-$ implies
\begin{equation}
\label{relation} t=t_0-\log\left(1+ |\textbf{x}-\ell|^2\right).
\end{equation}
In particular, the minimality of $t\left(k\right)$ implies that
$$t\left(k\right)= r_{\ell}-\log\left( 1+ |\textbf{x}-\ell|^2\right)$$
for some $\textbf{x}$ (determined by $k$). As $k$ ranges over $K(\ell)$ (that is, $\textbf{x}$ ranges over $B^{-}$), $t\left(k\right)$ attains the maximal value when $\textbf{x}=\frac12$ and the minimal value when $\textbf{x}= -\frac12$. Let $t_{\ell,\pm \frac12}= r_{\ell}-\log\left( 1+ |\ell\mp \frac12|^2\right)$, then in view of $\left(\ref{estimate set}\right)$ we have
$$A\left(t_{\ell,\frac12}\right)\times K\left(\ell\right)\subset Q\backslash QA\left(r_{\ell}\right)B^{-}g_{-\ell}\subset A\left(t_{\ell,-\frac12}\right)\times K\left(\ell\right).$$
Next, note that $e^{-nt_{\ell,\frac12}}\asymp e^{-nt_{\ell, -\frac12}}\asymp e^{-nr_{\ell}}\ell^{2n}$, hence
$$\left|A\left(t_{\ell,\frac12}\right)\times K\left(\ell\right)\right|\asymp \left|Q\backslash QA\left(r_{\ell}\right)B^{-}g_{-\ell}\right|\asymp \left|A\left(t_{\ell,-\frac12}\right)\times K\left(\ell\right)\right|.$$
Finally, note that $t_{\ell, \frac12}\leq \tau_{\ell}$ and $e^{-nt_{\ell,\frac12}}\asymp e^{-n\tau_{\ell}}$, hence \begin{displaymath}
A\left(\tau_{\ell}\right)\times K\left(\ell\right)\subset Q\backslash QA\left(r_{\ell}\right)B^{-}g_{-\ell}\quad \textrm{and}\quad |A\left(\tau_{\ell}\right)\times K\left(\ell\right)|\asymp |Q\backslash QA\left(r_{\ell}\right)B^{-}g_{-\ell}|.\qedhere
\end{displaymath}
\end{proof}
\begin{Rem}
\label{rem} Later we will take $r_{\ell}=\frac{1-\epsilon}{n}\log \ell$ with $\epsilon$ some fixed small positive number. We note that in this case we can take $p\left(m\right)=2m$. Moreover, since $\tau_{m}\geq \tau_{\ell}$ and $e^{-n\tau_{m}}\asymp e^{-n\tau_{\ell}}\asymp m^{(2n-1+\epsilon)}$ for all $m\leq \ell\leq 2m$, in view of Lemma \ref{lem 4} we have
$$A\left(\tau_{m}\right)\times K_m\subset \mathfrak{D}_m\quad\textrm{and}\quad |A\left(\tau_{m}\right)\times K_m|\asymp |\mathfrak{D}_m|,$$
where $K_m:=\cup_{\ell=m}^{2m}K\left(\ell\right)$.
\end{Rem}

\subsection{Proof of Theorem \ref{main thm}.}
Now we can give the proof of logarithm laws.
\subsection*{Upper bound}
Fix $\epsilon> 0$ and let $r_{\ell}= \frac{1+\epsilon}{n}\log\left(\ell\right)$. By $\left(\ref{norme}\right)$ the sets
$$\{x\in\Gamma\backslash G\ |\ xg_{\ell}\in B_{r_{\ell}}\}= B_{r_{\ell}}g_{-\ell}$$ satisfy
$$\sum_{\ell=1}^{\infty}\sigma\left(B_{r_{\ell}}g_{-\ell}\right)=\sum_{\ell=1}^{\infty}\sigma\left(B_{r_{\ell}}\right)\asymp \sum_{\ell=1}^{\infty}\frac{1}{\ell^{1+\epsilon}}< \infty.$$
Hence by Borel-Cantelli lemma the set
$$\{x\in\Gamma\backslash G\ |\ xg_{\ell}\in B_{r_{\ell}}\ \textrm{for finitely many $\ell$}\}$$
has full measure. This implies that
$$\limsup_{\ell\to\infty}\frac{\textrm{dist}\left(o,xg_{\ell}\right)}{\log \ell}\leq \frac{1+ \epsilon}{n}$$
for $\sigma$-a.e. $x\in \Gamma\backslash G$.
Moreover, for all $t\in \mathbb{R}$ let $\ell=\floor{t}$, we have
$$|\textrm{dist}\left(o, xg_t\right)-\textrm{dist}\left(o, xg_{\ell}\right)|\leq \textrm{dist}\left(xg_t, xg_{\ell}\right)\leq \textrm{dist}_G\left(e, g_{\ell-t}\right)= O\left(1\right),$$
hence we can replace the discrete limit over $\ell\in \mathbb{N}$ with a continuous limit over $t\in \bbR$. Finally, letting $\epsilon\to 0$ we get
$$\limsup_{t\to \infty}\frac{\textrm{dist}\left(o, xg_t\right)}{\log t}\leq 1$$
for $\sigma$-a.e. $x\in \Gamma\backslash G$.

\subsection*{Lower bound} Fix $\epsilon> 0$ and let $r_{\ell}= \frac{1-\epsilon}{n}\log \ell$. Let $\mathfrak{D}_m$ and $Y_{\mathfrak{D}_m}$ be as above. Note that in this case, for the definition of $\mathfrak{D}_m$ we can take $p\left(m\right)=2m$. We first prove the following

\begin{Lem}
\label{lem 4} There is a constant $\kappa_{\Gamma}>0$ depending only on $\Gamma$ such that $\sigma\left(Y_{\mathfrak{D}_m}\right)\geq \kappa_{\Gamma}$ for all $m\geq 1$.
\end{Lem}
\begin{proof}
Let $\tau_{m}=r_{m}-2\log\left(m\right)+\log 2$ be as above. Let $T$ be a sufficiently large integer such that
\begin{equation}
\label{approx}\frac{1}{2n} e^{-n\tau_{m}}\leq \int_{\tau_{m}}^{T}e^{-nt}dt\leq \frac{2}{n} e^{-n\tau_{m}}\ \textrm{and}\ \frac{1}{2n} e^{-s\tau_{m}}\leq \int_{\tau_{m}}^{T}e^{-st}dt\leq \frac{4}{n}e^{-s\tau_{m}}
\end{equation}
for any $s\in \left(\frac{n}{2},n\right)$ and $m\geq 1$. We identify the subgroup $A$ with $\bbR$ by sending $a_t$ to $t$. For every $m\geq 1$ define the set
$$\mathfrak{D}_m':=[\tau_{m}, T]\times K_m,$$
where $K_m=\cup_{\ell=m}^{2m}K\left(\ell\right)$.
By Remark $\ref{rem}$ we have $\mathfrak{D}_m'\subset \mathfrak{D}_m$ and $|\mathfrak{D}_m'|\asymp |\mathfrak{D}_m|$. In particular,  by Lemma \ref{lem 2}, $|\mathfrak{D}_m'|\gtrsim 1$ are uniformly bounded from below for all $m\geq 1$. Using the same unfolding trick as we did in Lemma \ref{unfold}, one has the following identity
$$\int_{\Gamma\backslash G}\Theta_{\mathbbm{1}_{\mathfrak{D}_m'}}\left(g\right)d\sigma\left(g\right)= \frac{\omega_{\Gamma}}{\nu_{\Gamma}}\int_{Q\backslash G}\mathbbm{1}_{\mathfrak{D}_m'}\left(g\right)dg$$
where $\omega_{\Gamma}=\int_{\Gamma_{\infty}\backslash Q}dq$ and $\nu_{\Gamma}=\int_{\Gamma\backslash G}dg$.
By Cauchy-Schwartz and the fact that $\Theta_{\mathbbm{1}_{\mathfrak{D}_m'}}$ is supported on $Y_{\mathfrak{D}_m'}$, we have
\begin{align*}
\left(\frac{\omega_{\Gamma}}{\nu_{\Gamma}}\right)^2|\mathfrak{D}_m'|^2&= \left(\int_{Y_{\mathfrak{D}_m'}}\Theta_{\mathbbm{1}_{\mathfrak{D}_m'}}\left(g\right)d\sigma\left(g\right)\right)^2\\
&\leq \sigma\left(Y_{\mathfrak{D}_m'}\right)||\Theta_{\mathbbm{1}_{\mathfrak{D}_m'}}||_2^2.
\end{align*}
Now in view of $\left(\ref{approx}\right)$ we can take $f_m=v_m\mathbbm{1}_{K_m}$ with $\mathbbm{1}_{K_m}$ the characteristic function of $K_m$ and $v_m$ approximating $\mathbbm{1}_{[\tau_{2m}, T]}$ from above sufficiently well such that
\begin{enumerate}
\item[(1)] $v_m$ is smooth, compactly supported and takes values in $[0,1]$;
\item[(2)] $\frac{1}{3n}\leq \frac{\int v_m\left(t\right)e^{-nt}dt}{e^{-n\tau_{m}}}, \frac{\int v_m^2\left(t\right)e^{-nt}dt}{e^{-n\tau_{m}}}\leq \frac{3}{n}$ and $\frac{1}{3n}\leq \frac{\int v_m\left(t\right)e^{-st}dt}{e^{-s\tau_{m}}}\leq \frac{5}{n}$ for any $s\in \left(\frac{n}{2},n\right)$;
\item[(3)] $||f_m||_1\leq 2|\mathfrak{D}_m'|$.
\end{enumerate}
In particular, for any $s\in\left(\frac{n}{2},n\right)$
$$\frac{\int v_m\left(t\right)e^{-st}dt}{\left(\int v_m\left(t\right)e^{-nt}dt\right)^{\frac{2s}{n}-1}\left(\int v_m^2\left(t\right)e^{-nt}dt\right)^{1-\frac{s}{n}}}\leq \frac{\frac{5}{n}e^{-s\tau_{m}}}{\left(\frac{1}{3n}e^{-n\tau_{m}}\right)^{\frac{2s}{n}-1+1-\frac{s}{n}}}<15.$$
Hence $\{f_m\}\subset \mathscr{A}_{15}$ and by Theorem \ref{thm 2} we can bound
$$||\Theta_{\mathbbm{1}_{\mathfrak{D}_m'}}||_2^2\leq ||\Theta_{f_m}||_2^2\lesssim_{\Gamma} ||f_m||_1^2+||f_m||_2^2.$$
Next, since $||f_m||_2^2\leq ||f_m||_1$, $||f_m||_1\leq 2|\mathfrak{D}_m'|$ and $|\mathfrak{D}_m'|\gtrsim 1$, we can bound
$$||f_m||_1^2+||f_m||_2^2\leq ||f_m||_1^2+ ||f_m||_1\leq 4|\mathfrak{D}_m'|^2+ 2|\mathfrak{D}_m'|\lesssim |\mathfrak{D}_m'|^2.$$
Finally, since  $Y_{\mathfrak{D}_m'}\subset Y_{\mathfrak{D}_m}$, we conclude that there is a constant $\kappa_{\Gamma}>0$ (independent of $m$) such that $\sigma\left(Y_{\mathfrak{D}_m}\right)\geq \sigma\left(Y_{\mathfrak{D}_m'}\right)\geq \kappa_{\Gamma}$ for any $m\geq 1$.
\end{proof}

Now consider the set $\mathcal{B}_{\epsilon}:= \cap_{\ell=L}^{\infty}\cup_{m=\ell}^{\infty}Y_{\mathfrak{D}_m}$, where $L$ is as in Lemma \ref{lem 1}. Then $\sigma\left(\mathcal{B}_{\epsilon}\right)\geq \kappa_{\Gamma}>0$ by Lemma $\ref{lem 4}$. Moreover, by Lemma $\ref{lem 1}$, for any $m\geq L, x\in Y_{\mathfrak{D}_m}$ there is some $\ell\geq m$ such that $\textrm{dist}\left(o, xg_{\ell}\right)\geq r_{\ell}+ O\left(1\right)$. Hence for any $x\in \mathcal{B}_{\epsilon}$ there is a sequence $\ell_m\to \infty$ such that $\textrm{dist}\left(o, xg_{\ell_m}\right)\geq r_{\ell_m}+ O\left(1\right)$. Consequently, we have
$$\mathcal{B}_{\epsilon}\subset \left\{x\in \Gamma\backslash G\ |\ \limsup\limits_{t\to \infty}\frac{\textrm{dist}\left(o, xg_t\right)}{\log t}\geq \frac{1-\epsilon}{n}\right\}.$$
Since the latter set is invariant under the action of $\{g_t\}_{t\in \bbR}$, by ergodicity it has full measure. Letting $\epsilon\to 0$ we get
$$\limsup\limits_{t\to \infty}\frac{\textrm{dist}\left(o, xg_t\right)}{\log t}\geq \frac{1}{n}$$
for $\sigma$-a.e. $x\in \Gamma\backslash G$.

\begin{Rem}
The same argument works for $r_{\ell}=\frac{1}{n}\log\ell+ O(1)$ with taking $p(m)= 2m$. Hence we can show that the sequence of nested cusp neighborhoods $\left\{B_{r_{\ell}}\right\}$ with $\sigma(B_{r_{\ell}})\asymp\frac{1}{\ell}$ is Borel-Cantelli for unipotent flows in this setting.
\end{Rem}

\end{document}